\documentclass[a4paper,10pt]{article}

\usepackage[utf8]{inputenc}
\usepackage[babel]{csquotes}
\usepackage{amsthm, amsmath, amssymb,amsfonts,latexsym}
\usepackage[scale=0.8]{geometry}  
\usepackage{bbm}       
\usepackage[version=3]{mhchem}    
\usepackage{tikz}      
\usetikzlibrary{shapes,automata,positioning,arrows,fit}
\usepackage{pgfplots}
\pgfplotsset{compat=1.7}
\usepackage{subcaption,enumerate}

\usepackage[numbers, sort]{natbib}

\usepackage[colorlinks]{hyperref}

\newcommand{\Sp}{\mathcal{S}}
\newcommand{\C}{\mathcal{C}}
\newcommand{\R}{\mathcal{R}}
\newcommand{\G}{\mathcal{G}}

\newcommand{\ZZ}{\mathbb{Z}}
\newcommand{\RR}{\mathbb{R}}

\newcommand{\scal}[2]{\langle#1,#2\rangle} 
\DeclareMathOperator{\spann}{span} 
\DeclareMathOperator*{\argmax}{arg\,max}

\newcommand{\tle}{\precsim_{(x_n)}} 
\newcommand{\tl}{\prec_{(x_n)}} 
\newcommand{\te}{\sim_{(x_n)}} 
\newcommand{\tlwrt}[1]{\prec_{(#1)}} 
\newcommand{\tewrt}[1]{\sim_{(#1)}} 

\newcommand{\ld}{\lambda^D} 
\newcommand{\ls}{\lambda^S} 

\def\S{\mathcal S}
\def\C{\mathcal C}
\def\Re{\mathcal R}

\def\Z{\mathbb Z}

\def\Td1{T^{D,1}_{\{x_n\}}}
\def\Ts1{T^{S,1}_{\{x_n\}}}

\def\lmt{\lim_{n\rightarrow \infty}}

\newcommand{\reply}[1]{{#1}}

\binoppenalty=\maxdimen
\relpenalty=\maxdimen

\numberwithin{equation}{section}

\newtheorem{theorem}{Theorem}[section]
\newtheorem{corollary}[theorem]{Corollary}
\newtheorem{lemma}[theorem]{Lemma}
\newtheorem{proposition}[theorem]{Proposition}

\theoremstyle{remark}
\newtheorem{remark}{Remark}[section]
\newtheorem{example}{Example}[section]

\theoremstyle{definition}
\newtheorem{definition}{Definition}[section]
\newtheorem{assumption}{Assumption}

\title{Tier structure of strongly endotactic reaction networks\footnote{This document was conceived and prepared while all the authors were working at the Department of Mathematics of the University of Wisconsin-Madison.} 
}
\author{David F.\ Anderson\footnote{Department of Mathematics, University of Wisconsin-Madison, \href{mailto:anderson@math.wisc.edu}{anderson@math.wisc.edu}.  Grant support from NSF-DMS-1318832 and Army Research Office W911NF-14-1-0401.} \and Daniele Cappelletti\footnote{Department of Biosystems Science and Engineering, ETH Z\"{u}rich, \href{mailto:daniele.cappelletti@bsse.ethz.ch}{daniele.cappelletti@bsse.ethz.ch}.} \and Jinsu Kim\footnote{Department of Mathematics, University of California, Irvine, \href{mailto:jinsu.kim@uci.edu}{jinsu.kim@uci.edu}.} \and Tung Nguyen\footnote{Department of Mathematics, University of Wisconsin-Madison, \href{mailto:nguyen34@wisc.edu}{nguyen34@wisc.edu}.}}

\begin{document}
	
	\maketitle
	
	\begin{abstract}
Reaction networks are mainly used to model the time-evolution of molecules of interacting chemical species.  Stochastic models are typically used when the counts of the molecules  are low, whereas deterministic models are often used when the counts are in high  abundance.  The mathematical study of reaction networks has increased dramatically over the last two decades as these models are now routinely used to investigate cellular behavior.   In 2011, the notion of ``tiers'' was introduced to study the long time behavior of deterministically modeled reaction networks
that are weakly reversible and have a single linkage class.  This ``tier'' based argument was analytical in nature.
Later, in 2014, the notion of a strongly endotactic network was introduced in order to generalize the previous results  from weakly reversible networks with a single linkage class to this wider family of networks.  The point of view of this later work was more geometric and algebraic in nature.   The notion of strongly endotactic networks was later used in 2018 to prove a large deviation principle for a class of stochastically modeled reaction networks.

 In the current paper we provide an analytical characterization of strongly endotactic networks in terms of tier structures. By doing so, we not only shed light on the connection between the two points of view, but also make available a new proof technique for the study of strongly endotactic networks.  We show the power of this new technique in two distinct ways.  First, we demonstrate how the main  previous results related to strongly endotactic networks, both for the deterministic and stochastic modeling choices, can  be quickly obtained from our characterization.  Second, we demonstrate how new results can be obtained by proving that a sub-class of strongly endotactic networks, when modeled stochastically, is positive recurrent.  Finally, and similarly to recent independent work by Agazzi and Mattingly, we provide an example which closes a conjecture in the negative by showing that stochastically modeled strongly endotactic networks can be transient (and even explosive).
	\end{abstract}
	
	
\section{Introduction}

Reaction networks are now commonly  used to model the dynamical behavior of cellular processes, including gene regulatory systems, signaling systems, metabolic systems,  viral infections, etc. If the counts of the constituent ``species'' of the system of interest are low, then the dynamics of the counts are typically modeled via a continuous-time Markov chain with state space $\Z^d_{\ge 0}$, where $d$ is the number of species in the system.  On the other hand, if the counts are high, then an autonomous system of (typically non-linear) ordinary differential equations in $\RR^d_{\ge0}$ is used to model the dynamics of the relevant chemical concentrations.  See \cite{AKbook,K:strong,K:stochastic} for the precise connection between these two modeling choices.

The mathematical foundation of deterministically modeled reaction networks can largely be traced back to the series of papers  \cite{F1,H,H-J1}, where Feinberg, Horn, and Jackson introduced the  notion of network deficiency and proved that if the reaction network  (i) is weakly reversibility and (ii) has a deficiency of zero, then the resulting deterministically modeled system endowed with mass-action kinetics is ``complex-balanced,'' which means that (i) every linear invariant manifold in $\RR^d_{>0}$ admits precisely one equilibrium point, and (ii)  each of these equilibria  satisfies a particular network balance and it is a so-called complex-balanced equilibrium.  Importantly, they showed that this result holds regardless of the choice of rate parameters for the model.  
Feinberg, Horn, and Jackson were interested in chemical systems at an industrial scale.  At smaller scales, discrete stochastic models have been utilized and studied.  In particular, the works of Gardiner \cite{gardiner1985stochastic}, Van Kampen \cite{VK81}, Kurtz \cite{K:strong,K:stochastic,Kur71}, and Gillespie \cite{Gill76,Gill77} were all instrumental in increasing our understanding of these models.  

Much of the work cited in the previous paragraph  took place in the 1960s and 1970s.  While there was mathematical work related to reaction networks during the 1980s and 1990s, it was the advent of new technologies in the biological setting, such as fluorescent proteins, that made the study of mathematical models of reaction networks quite popular  over the last two decades.   

Reaction networks can naturally be associated with a finite graph, constituted by the set of all chemical reactions that can take place. For a few examples of such graphs, see Examples  \ref{ex:SE_transient}, \ref{ex:SE_explosive}, and \ref{ex:proper_not_transversal}. Much of the theory on reaction networks deals with connections between such finite graphs, which are relatively easy to study, and the qualitative properties of the associated dynamical system. We note also that it is most useful to provide results that hold for any choice of model parameters, as these parameters are often unknown with any certainty  in the biological setting. Specifically, the mathematical results about reaction networks are often of the following form:  
\vspace{.1in}

\textit{
Consider a reaction network whose associated graph has properties A, B, and C.  Then, for any choice of parameters for the model, the relevant dynamical system will have property D.}

\vspace{.1in}
  \noindent For example, in the works of Horn, Jackson, and Feinberg cited above, weak reversibility and a deficiency of zero  are both structural properties of the graph, and they imply qualitative dynamical properties of the models such as non-chaotic behavior of the trajectories and the absence of limit cycles, regardless of the choice of model parameters. 
    
 For our purposes, the  most relevant previous works in the field  are    \cite{A:boundedness, A:single} by Anderson, \cite{GMS:geometric} by Gopalkrishnan, Miller, and Shiu, and  \cite{ADE:deviation,ADE:geometric}  by Agazzi, Dembo, and Eckmann.  In   \cite{A:boundedness,A:single}, Anderson developed the concept of ``tiers'' of complexes, and used them to study deterministically modeled reaction networks. Loosely speaking, tiers constitute a partition of the system complexes (see section \ref{sec:background} for relevant definitions) into sets related to reactions whose propensities have the same relative strength along a particular sequence of points in $\RR^d$.   The  works  \cite{A:boundedness, A:single} used tiers to prove that trajectories for reactions networks that were (i) weakly reversible and (ii) had a single linkage class, were necessarily   persistent (meaning that they cannot get arbitrarily close to the boundary of the state space, see Definition \ref{def:persistence}) and bounded, regardless of the choice of rate parameters.  These works closed the well-known Global Attractor Conjecture in the single linkage class case \cite{CDSS2009}.   Later, in  \cite{GMS:geometric}, Gopalkrishnan, Miller, and Shiu (i) introduced the notion of strongly endotactic networks (which are a subclass of \textit{endotactic networks}, introduced in \cite{CNP2013}), (ii) showed that weakly reversible networks that have  a single linkage class are strongly endotactic, and (iii) showed that deterministically modeled strongly endotactic networks are permanent (which  is a stronger condition than persistence and boundedness of trajectories, see Definition \ref{def:permanence}).   The main results of \cite{GMS:geometric} are stated here as Theorems \ref{thm:persistence} and \ref{thm:permamence}. Finally, the class of strongly endotactic networks have been fruitfully recently considered in \cite{ADE:deviation,ADE:geometric}, where a large deviation principle for stochastically modeled reaction networks that are strongly endotactic and that are also ``asiphonic'' is provided.
 
 The tier argument developed in \cite{A:single,A:boundedness} was analytical in nature, whereas the methods developed in \cite{GMS:geometric} and later utilized in \cite{ADE:deviation,ADE:geometric}, while quite similar to those developed in \cite{A:single,A:boundedness}, were more algebraic and geometric in nature.  In the present work, we will make the connections between the two works more precise.  Specifically, we will characterize strongly endotactic networks in regards to their tier structures.  
 
 Elucidating the connection between strongly endotactic networks and tiers is the first major contribution of this work, and provides a new proof technique for the study of strongly endotactic networks.     We will demonstrate the power of this new technique in two distinct ways. 
 \begin{enumerate}
 \item  We  show how the proofs of the major results related to strongly endotactic networks in both the deterministic and stochastic settings can be dramatically streamlined.   First, we will show how the main results of \cite{GMS:geometric} related to deterministic models of reaction networks that are strongly endotactic  follow in a straightforward manner by the tier characterization.  Second, we will show how the main analytical results of \cite{ADE:deviation,ADE:geometric} can be quickly recovered using our characterization.  
 
 \item We show that members of a particular subclass of strongly endotactic networks are positive recurrent when modeled stochastically, regardless of the choice of rate parameters.  
 \end{enumerate}
 We make one further contribution in this paper. It has been proven in a number of instances  that the behaviors of the associated deterministic and stochastic models for reaction networks are similar in a broad sense.  For example, there is theory connecting the dynamics of the two models on compact time intervals  \cite{K:stochastic, K:strong, ACK:ACR}, on pathwise approximations \cite{CW:st_intermediates, CW:det_intermediates}, and on similarities between their long time stationary behavior  \cite{ACK:product, ACKK:explosion, CW:CB, CJ:balance}. Hence, since it is proven in \cite{GMS:geometric} that deterministically modeled strongly endotactic networks have very well behaved trajectories in the sense made precise by Theorems \ref{thm:persistence} and \ref{thm:permamence}, it was natural to conjecture that all strongly endotactic networks are necessarily positive recurrent when stochastically modeled.  We show this conjecture to be false by providing strongly endotactic networks that are transient and even explosive, regardless of the choice of parameters for the model (see Examples \ref{ex:SE_transient} and \ref{ex:SE_explosive}).    (We note  that the  conjecture has independently been shown to be false in the recently submitted  paper \cite{AM2018} by Mattingly and Agazzi, where other examples are provided.)
 
The outline of the remainder of the paper is as follows.  In section \ref{sec:background}, we provide useful notation, and the relevant mathematical models.  In section \ref{sec:SE_networks}, we provide the definition of a strongly endotactic network. We also provide the examples alluded to in the previous paragraph demonstrating that not all strongly endotactic networks are recurrent, when modeled stochastically.
In section \ref{sec:tiers}, we provide the relevant definitions pertaining to tiers.  In particular, in subsection \ref{sec:SE_tiers} we provide our  main analytical result, Theorem \ref{thm:SE_tiers}, that characterizes strongly endotactic networks by their tier structures.  In section \ref{sec:permanence}, we use our results from section \ref{sec:tiers} to prove that deterministically modeled strongly endotactic networks are both persistent and permanent.  Therefore, the results of section \ref{sec:permanence} recover the main findings in \cite{GMS:geometric}.   In section \ref{sec:LDP}, we utilize our results from section \ref{sec:tiers} to recover \reply{a sufficient condition used in \cite{ADE:deviation,ADE:geometric} to prove a large deviation principle}. Finally, in section \ref{sec:PR_SE}, we use the results of section \ref{sec:tiers} to provide a new  subclass of reaction networks for which  positive recurrence is guaranteed, regardless of the choice of rate parameters.

\section{Background}
\label{sec:background}

\subsection{Notation}
Throughout the paper, we will denote by $\RR$, $\RR_{\geq0}$, and $\RR_{>0}$ the real, the non-negative real, and the positive real numbers, respectively. Similarly, we will denote by $\ZZ$, $\ZZ_{\geq0}$, and $\ZZ_{>0}$ the integer, the non-negative integer, and the positive integer numbers, respectively.
Given a vector $v\in\RR^d$, we say that the vector is \emph{positive} or \emph{non-negative} if $v$ is in $\RR^d_{>0}$ or $\RR^d_{\geq0}$, respectively.

Given two vectors $v,w\in\RR^d$, we will denote by $\scal{v}{w}$ their scalar product. Furthermore, we will write $v\geq w$ if the inequality holds component-wise. Moreover, we will use the following shorthand notation:
$$v^w=\prod_{i=1}^d v_i^{w_i},\quad v!=\prod_{i=1}^d v_i!,$$
where we use the usual convention $0^0 = 1$.
Finally, we will denote by $\ln(v)$ the vector of $\RR^d$ whose $i$th entry is $\ln(v_i)$ and we will denote by $\lfloor v\rfloor$ the vector whose $i$th entry is $\lfloor v_i \rfloor$.

Given a vector $v\in\RR^d$, we denote
$$\|v\|_\infty=\max\{|v_i|\,:\, 1\leq i\leq d\} \quad \text{and} \quad \|v\|_1 = \sum_{i=1}^d |v_i|.$$
Moreover, we denote by $v \vee 1$ the vector whose $i$th component is $\max\{v_i,1\}$.
For two sequences of positive real numbers $(a_n)_{n=0}^\infty$ and $(b_n)_{n=0}^\infty$, we write $a_n \gg b_n$ if $\lim_{n\to \infty} \frac{a_n}{b_n} = \infty$.

\subsection{Reaction networks}
A \emph{reaction network} is a triple $\G=(\Sp,\C,\R)$ where $\Sp$, $\C$, and $\R$ are defined as follows. $\Sp$ is a finite set of \emph{species}, that is a set of $d$ distinct symbols.  $\C$ is a finite set of \emph{complexes}.
 We assume each complex is a linear combinations of species on $\ZZ_{\geq0}$. Complexes will be regarded as vectors in $\ZZ_{\geq0}^d$ in the paper, given that an ordering for the species is chosen. Finally, $\R$ is a finite set of \emph{reactions}, that is a finite subset of $\C\times\C$ with the property that for any $y\in\C$ we have $(y,y)\notin\R$. Usually, a reaction $(y,y')$ is denoted by $y\to y'$, and we adopt this notation in the paper.

We say that $y$ is a \emph{source complex} if there is a reaction of the form $y\to y'$, and we say that $y$ is a \emph{product complex} if there is a reaction of the form $y'\to y$. Moreover, given a reaction $y\to y'$ we say that $y$ is the source and $y'$ is the product of $y\to y'$.

It is often convenient to denote the species as $\{S_1,\dots, S_d\}$, as this allows us to refer to species via their index.  In particular, we will  write both $S_i\in \Sp$ and $i \in \Sp$.  However, in practical examples the set of species is often given as some subset of $\{A,B,C,\dots\}$. 

Given a reaction network $\G$, a directed graph with nodes $\C$ and edges $\R$ can be constructed. This directed graph is called \emph{reaction graph}. See Examples  \ref{ex:SE_transient}, \ref{ex:SE_explosive}, and \ref{ex:proper_not_transversal} for examples of such graphs. In this paper we assume that all complexes appear in at least one reaction and all species appear in at least one complex. Under this assumption, the reaction graph uniquely determines a reaction network. In fact, reaction networks are usually described by means of their reaction graph, and the same will be done in the present paper.

The \emph{stoichiometric subspace} is defined as
$$S=\spann_{\RR}\{y'-y\,:\,y\to y'\in\R\},$$
and for any $x\in \RR^d$ the set $x + S = \{x + s, \text{ with } s \in S\}$ is termed the stoichiometric compatibility class determined by $x$.  Similarly, the sets $(x+S) \cap \RR^d_{\ge 0}$ are the nonnegative stoichiometric compatibility classes.
\subsection{Deterministic model}
Deterministic models are typically used when the counts of the relevant molecules (the species) are large and their concentrations change nearly continuously in time  accordingly to the propensities of the different chemical transformations.

Formally, given a reaction network $\G$, a \emph{(deterministic) kinetics} $\Lambda$ for $\G$ is a map assigning a function $\lambda_{y\to y'}:\RR_{\geq0}^d\to \RR_{\geq0}$ to each reaction $y\to y'\in \R$. The functions $\lambda_{y\to y'}$ are called \emph{(deterministic) rate functions}.
We call  a pair $(\G,\Lambda)$, where $\G$ is a reaction network and $\Lambda$ is a deterministic kinetics, a \emph{deterministic reaction system}. \reply{In this setting, the concentration of the different chemical species is of interest, which should be understood as the average number of molecules of the different species per unit of volume. If molecules are approximated with dimensionless points in space, concentrations are non-negative real numbers ranging from 0 to $\infty$.} Given an initial condition $z(0)\in\RR^d_{\geq0}$, the change in chemical species concentration is then modeled as the solution to the  integral equation 
\begin{equation}\label{eq:ODE}
 z(t)=z(0)+\sum_{y\to y'\in\R} (y'-y)\int_0^t\lambda_{y\to y'}(z(s))\,ds,
\end{equation}
\reply{if the solution exists and is unique.}
Note that at any time point $t$, $z(t)- z(0)\in S$. That is, $z(t)$ is confined within the same stoichiometric compatibility class as $z(0)$.

A popular choice of kinetics is given by \emph{(deterministic) mass action kinetics}, where for any reaction $y\to y'\in\R$\reply{, the associated rate function is given by}
$$\ld_{y\to y'}(x)=\kappa_{y\to y'} x^y,$$
for some positive constant $\kappa_{y\to y'}$, termed a \emph{reaction constant}. \reply{Note that under the assumption of mass action kinetics,  the solution to \eqref{eq:ODE} exists and is unique for any initial condition, since the rates $\lambda_{y\to y'}$ are polynomials and therefore locally Lipschitz. In contrast, global existence is not guaranteed, and in case of a blow-up at a finite time $t^\star$ we consider the solution to \eqref{eq:ODE} only in the interval $[0,t^\star)$.} 

Mass action kinetics corresponds to the hypothesis that the molecules of the chemical species involved in the transformations are well-stirred. \reply{In the present paper, we will focus on this choice of kinetics, which is typically studied in reaction network theory \cite{erdi1989mathematical} and in biochemistry \cite{edelstein2005mathematical,ingalls2013mathematical}. More general kinetics (such as Michaelis-Menten kinetics) can be derived as limits of mass action kinetics when different chemical reactions operate over time scales of different orders of magnitude \cite{edelstein2005mathematical,ingalls2013mathematical}.}

\subsection{Stochastic model}\label{sec:stochastic_model}
Stochastic models are typically used when we are interested in the counts of the different chemical species. This situation typically arises when the abundances are low, as is often the case in the biological setting. 

The formal definition of stochastic reaction systems follows  the definition of deterministic reaction systems closely: given a reaction network $\G$, a \emph{(stochastic) kinetics} $\Lambda$ for $\G$ is a map assigning a function $\lambda_{y\to y'}:\ZZ_{\geq0}^d\to \RR_{\geq0}$ to each reaction $y\to y'\in \R$. The functions $\lambda_{y\to y'}$ are called \emph{(stochastic) rate functions}.
A \emph{stochastic reaction system} is a pair $(\G,\Lambda)$, where $\G$ is a reaction network and $\Lambda$ is a stochastic kinetics. The change in chemical species counts is modeled by means of a continuous-time Markov chain with state space $\ZZ_{\geq0}^d$, whose transition rates are given by
$$q(x,x+\xi)=\sum_{\substack{y\to y'\in\R\\ y'-y=\xi}}\lambda_{y\to y'}(x).$$
\reply{In case of an explosion occurring at a finite-time $T_\infty$, we consider the $X(t)=\Delta$ for any $t\geq T_\infty$, where $\Delta$ is a cemetery state not contained in $\ZZ_{\geq0}^d$.} Another representation of the Markov chain $X$, due to Kurtz \cite{Kurtz80},  is given as follows:
\begin{align}\label{eq:stochastic model}
X(t)=X(0)+\sum_{y\to y'\in\R} (y'-y)Y_{y\to y'}\left(\int_0^t\lambda_{y\to y'}(X(s))\,ds\right)
\end{align}
where $Y_{y\to y'}$ are independent, unit-rate Poisson processes.  Letting $T_n$ denote the time of the $n$th transition of the model, the above representation is valid up until $T_\infty = \lim_{n\to \infty} T_n$. Here, the counting process
$$R_{y\to y'}(t)=Y_{y\to y'}\left(\int_0^t\lambda_{y\to y'}(X(s))\,ds\right)$$
keeps track of how many times the reaction $y\to y'$ has occurred by time $t$.

From \eqref{eq:stochastic model} we  have that  $X(t)-X(0)\in S$ for any time point $t$.  Hence, and just as for the deterministic model, the stochastic process $X$ is confined within the stoichiometric compatibility class determined by $X(0)$.  

A popular choice of stochastic kinetics is given by  \emph{(stochastic) mass action kinetics}, where for any reaction $y\to y'\in\R$
\begin{align}\label{eq:stochastic mass action}
\ls_{y\to y'}(x)=\kappa_{y\to y'} \mathbbm{1}_{\{x\geq y\}}\frac{x!}{(x-y)!},
\end{align}
for some positive constant $\kappa_{y\to y'}$, called a \emph{reaction constant}. Similarly with  deterministic reaction networks, mass action kinetics corresponds to the hypothesis that the molecules are well-stirred in space. \reply{The analysis of the present paper focus on this choice of kinetics.}

\section{Strongly endotactic networks}
\label{sec:SE_networks}
We give here the definition of strongly endotactic networks, that was first introduced in \cite{GMS:geometric}.

\begin{definition}
 Consider a reaction network $\G$, and a vector $w\in\RR^d$ that is not orthogonal to the stoichiometric subspace $S$. We say that a complex $y\in\C$ is \emph{$w-$maximal} if $y$ is a source complex and for any other source complex $y'$ we have $\scal{w}{y'-y}\leq 0$.
\end{definition}

\begin{definition}\label{def:strongly_endotactic}
 A reaction network $\G$ is \emph{strongly endotactic} if for all vectors $w\in\RR^d$ that are not orthogonal to the stoichiometric subspace $S$ the following holds:
 \begin{enumerate}
  \item if $y$ is a $w-$maximal complex, then for all reactions of the form $y\to y'$ we have $\scal{w}{y'-y}\leq 0$;
  \item there exists a $w-$maximal complex $y$ and a reaction $y\to y'\in\R$ with $\scal{w}{y'-y}<0$.
 \end{enumerate}
\end{definition}

Strongly endotactic networks are a generalization of weakly reversible single linkage class networks studied in \cite{A:single}: the following proposition, which \reply{is corollary 3.20 in \cite{GMS:geometric}}, makes the statement precise.

\begin{proposition}\label{prop:WR_SE}
 Assume $\G$ is a reaction network such that for any two complexes $y, y'$ there exists a sequence of $\ell$ complexes, $y=y_1, y_2, \dots, y_\ell=y'$, such that $y_j\to y_{j+1}\in\R$ for all $1\leq j\leq \ell-1$ (this condition is equivalent to saying that $\G$ is \emph{weakly reversible and consists of a single linkage class}). Then, $\G$ is strongly endotactic.
\end{proposition}

Strongly endotactic network are not necessarily weakly reversible single linkage class networks, examples are provided in Examples \ref{ex:SE_transient} and \ref{ex:SE_explosive}. 
As discussed in the Introduction, due to the stable behavior of the deterministic mass action systems associated with strongly endotactic networks (see Theorems \ref{thm:persistence} and \ref{thm:permamence}), it was conjectured that stochastic mass action systems associated to strongly endotactic networks would be positive recurrent for any choice of rate constants. This is not the case: in Example \ref{ex:SE_transient} a strongly endotactic network is considered that results in a transient system if endowed with stochastic mass action kinetics, for any choice of rate constants. Furthermore, in Example \ref{ex:SE_explosive} we show that a similar model is explosive for any choice of rate constants. 

\begin{example}\label{ex:SE_transient}
 Consider the reaction network 
 $$0\to 2A+B\to 4A+4B \to A.$$
 The reaction network is strongly endotactic: to check that this statement is true, it is convenient to draw the complexes considered as vectors on a Cartesian plane, and depict the reactions as arrows among them. This is done in Figure \ref{fig:A}. Now consider the shaded regions of Figure \ref{fig:B}: it can be checked that
 \begin{itemize}
  \item If $w\in R_1$, then the $w-$maximal complex is $4A+4B$. The only reaction with source complex $4A+4B$ is $4A+4B\to A$, and we have $\scal{w}{(-3,-4)}<0$.
  \item If $w\in R_2$, then the $w-$maximal complex is $0$. The only reaction with source complex $0$ is $0\to 2A+B$, and we have $\scal{w}{(2,1)}<0$.
  \item If $w\in R_3$, then the $w-$maximal complex is $2A+B$. The only reaction with source complex $2A+B$ is $2A+B\to 4A+4B$, and we have $\scal{w}{(2,3)}<0$.
  \item If $w$ is a positive multiple of $(-1,1)$, then the $w-$maximal complexes are $0$ and $4A+4B$, which are source complexes of $0\to 2A+B$ and $4A+4B\to A$. In this case, we have $\scal{w}{(2,1)}<0$ and $\scal{w}{(-3,-4)}<0$.
  \item If $w$ is a positive multiple of $(1,-2)$, then the $w-$maximal complexes are $0$ and $2A+B$, which are source complexes of $0\to 2A+B$ and $2A+B\to 4A+4B$. In this case, we have $\scal{w}{(2,1)}=0$ and $\scal{w}{(2,3)}<0$.
  \item If $w$ is a positive multiple of $(1,-2/3)$, then the $w-$maximal complexes are $2A+B$ and $4A+4B$, which are source complexes of $2A+B\to 4A+4B$ and $4A+4B\to A$. In this case, we have $\scal{w}{(2,3)}=0$ and $\scal{w}{(-3,-4)}<0$.
 \end{itemize}
 Hence, the network is strongly endotactic. A general strategy to recognize strongly endotactic network, called the \emph{sweep test}, and which we essentially carried out here in detail, is discussed in \cite{GMS:geometric}.

 \begin{figure}[h]
 \captionsetup[subfigure]{position=b}
 \centering
 \subcaptionbox{The complexes of the network of Example \ref{ex:SE_transient}, considered as vectors, are drawn. The reactions are represented by arrows. The shaded region represents the convex hull generated by the source complexes. Note that all reactions originated on the faces of the convex hull point inside the hull.\label{fig:A}}{
   \begin{tikzpicture}
    \begin{axis}[
    axis y line=middle,axis x line=middle,
      xmin=-0.5, xmax=4.8, ymin=-0.5, ymax=4.8,
      unit vector ratio*=1 1 1, 
      ticks=none, 
      xlabel={$A$}, ylabel={$B$},
      width=0.6\textwidth]
     \addplot[mark=none,gray, draw opacity=0.2, fill=gray, fill opacity=0.2]
             coordinates{
                          (0,0)
                          (2,1)
                          (4,4)
                          (0,0)
                          };
     \addplot[mark=none, black, -triangle 45]
              coordinates{
                          (0,0)
                          (2,1)
                          };
     \addplot[mark=none, black, -triangle 45]
              coordinates{
                          (2,1)
                          (4,4)
                          };
     \addplot[mark=none, black, -triangle 45]
              coordinates{
                          (4,4)
                          (1,0)
                          };
     \node[label={135:{0}},inner sep=2pt] at (axis cs:0,0) {};
     \node[label={0:{$2A+B$}},inner sep=2pt] at (axis cs:2,1) {};
     \node[label={90:{$4A+4B$}},inner sep=2pt] at (axis cs:4,4) {};
     \node[label={270:{$A$}},inner sep=2pt] at (axis cs:1,0) {};
    \end{axis}
   \end{tikzpicture}}  
   \hspace*{1em}
   \subcaptionbox{The space is divided into the open regions $R_1$, $R_2$, and $R_3$, which correspond to the loci of vectors $w$ with different $w-$maximal complexes, and into the rays separating them (which are orthogonal to the faces of the convex hull generated by the source complexes). The vectors $w$ laying on the separating lines have two $w-$maximal complexes. \label{fig:B}}{
   \begin{tikzpicture}
    \begin{axis}[
    axis y line=middle,axis x line=middle,
      xmin=-2.5, xmax=2.5, ymin=-2.5, ymax=2.5,
      unit vector ratio*=1 1 1, 
      ticks=none, 
      xlabel={$A$}, ylabel={$B$},
      width=0.6\textwidth]
      \addplot[mark=none, gray, draw opacity=0.3, fill=gray, fill opacity=0.3]
             coordinates{
                          (0,0)
                          (-6,6)
                          (6,6)
                          (6,-4)
                          (0,0)
                          };
      \addplot[mark=none, gray, draw opacity=0.6, fill=gray, fill opacity=0.6]
             coordinates{
                          (0,0)
                          (-6,6)
                          (-6,-6)
                          (6,-12)
                          (0,0)
                          };
      \addplot[mark=none, gray, draw opacity=0.9, fill=gray, fill opacity=0.9]
             coordinates{
                          (0,0)
                          (6,-12)
                          (6,-4)
                          (0,0)
                          };
      \addplot[mark=none, domain=-2.5:0] {-x};
      \addplot[mark=none, domain=0:2.5] {-2*x};
      \addplot[mark=none, domain=0:2.5] {-2*x/3};
      \node[label={180:{(-1,1)}},circle,fill, inner sep=2pt] at (axis cs:-1,1) {};
      \node[label={180:{(1,-2)}},circle,fill, inner sep=2pt] at (axis cs:1,-2) {};
      \node[label={0:{(1,-2/3)}},circle,fill, inner sep=2pt] at (axis cs:1,-2/3) {};
      \node[label={0:{$R_1$}},inner sep=2pt] at (axis cs:1,1) {};
      \node[label={180:{$R_2$}},inner sep=2pt] at (axis cs:-1,-1) {};
      \node[label={0:{$R_3$}},inner sep=2pt] at (axis cs:1.5,-2) {};
    \end{axis}
   \end{tikzpicture}}
\end{figure}

Nevertheless, any stochastic mass action system associated with the network is transient. Indeed, from any state $x=(x_A, x_B)\in\ZZ^d$ there is a positive probability that the reaction $0\to 2A+B$ occurs $j$ consecutive times, with $x_A+2j\geq x_B+j$ and $x_B+j$ being divisible by 4. There is then a positive probability that the reaction $4A+4B\to 0$ takes place until no molecule of $B$ is left, and a state of the form $x'=(x'_A,0)$ is reached. Then, due to continuity of probability measures, the probability, $p(x_A'),$ that from the state $x'$ the infinite repetition of the sequence of reactions $0\to 2A+B$, $2A+B\to 4A+4B$ and $4A+4B\to A$ take place is 
\begin{align*}
 p(x'_A)=\prod_{n=x'_A+1}^\infty 1&\cdot\frac{\kappa_{2A+B\to 4A+4B}(n+1)n}{\kappa_{2A+B\to 4A+4B}(n+1)n+\kappa_{0\to 2A+B}}\\
 &\cdot \frac{\kappa_{4A+4B\to A}(n+3)(n+2)(n+1)n\cdot 4!}{\kappa_{4A+4B\to A}(n+3)(n+2)(n+1)n\cdot 4!+\kappa_{2A+B\to 4A+4B}(n+3)(n+2)\cdot 4+\kappa_{0\to 2A+B}}.
\end{align*}
An infinite product of the form $\prod_n a_n b_n$, where $a_n,b_n \in (0,1)$, will converge to a nonzero value if and only if the infinite sum $\sum_n \left[ (1-a_n) + (1-b_n)\right]$ converges; see   \cite[Theorem 15.4]{Rudin:real_and_complex}.  The sum
\begin{align*}
 \sum_{n=x'_A+1}^\infty &\Bigg( \frac{\kappa_{0\to 2A+B}}{\kappa_{2A+B\to 4A+4B}(n+1)n+\kappa_{0\to 2A+B}}+\\
&+\frac{\kappa_{2A+B\to 4A+4B}(n+3)(n+2)\cdot 4+\kappa_{0\to 2A+B}}{\kappa_{4A+4B\to A}(n+3)(n+2)(n+1)n\cdot 4!+\kappa_{2A+B\to 4A+4B}(n+3)(n+2)\cdot 4+\kappa_{0\to 2A+B}}\Bigg)<\infty,
\end{align*}
has terms of order $n^{-2}$, and so converges.  Thus, we may conclude that $p(x'_A)>0$.
Hence, it follows that there is a positive probability of leaving the state $x$ forever through the repetition of the sequence of reactions $0\to 2A+B$, $2A+B\to 4A+4B$ and $4A+4B\to A$, which increases the number of molecules of $A$ at each cycle. It follows that  every state is transient, independently on the choice of positive rate constants.  
\hfill $\square$
\end{example}

We now show how a slight modification of the previous example leads to a strongly endotactic network that explodes for any initial condition.

\begin{example}\label{ex:SE_explosive}
 By modifying the reaction network in Example \ref{ex:SE_transient} we obtain
 $$0\to2A\to 4A+B\to 6A+4B \to 3A.$$
 The network is still strongly endotactic, as can be checked by utilizing a similar techinque as in Example \ref{ex:SE_transient}. Moreover, and similarly as in Example \ref{ex:SE_transient}, it can be verified that from any state $x\in\ZZ^d$ it is possible to reach a state of the form $x'=(x'_A,0)$, and by letting the reaction $0\to 2A$ take place we may assume that $x'_A\geq 2$. There is then a positive probability that starting from $x'$ the infinite repetition of the sequence of reactions $2A \to 4A+B$, $4A + B  \to 6A + 4B$, and  $6A + 4B \to 3A$ 
 occurs, each cycle increasing the number of molecules of $A$ by 1.
 The main difference with the previous example is that by the monotone convergence theorem the expected time it takes for the infinite repetition of the reaction sequence $2A \to 4A+B$, $4A + B  \to 6A + 4B$, and  $6A + 4B \to 3A$ to take place, $m(x_A')$, is bounded by
 \begin{align*}m(x'_A)<\sum_{n=x'_A}^\infty &\Bigg(\frac{1}{\kappa_{2A\to 4A+B}n(n-1)}+\frac{1}{\kappa_{4A+B\to 6A+4B}(n+2)(n+1)n(n-1)}\\
 &+\frac{1}{\kappa_{6A+4B\to 3A}(n+4)(n+3)(n+2)(n+1)n(n-1)\cdot 4!}\Bigg)<\infty,
 \end{align*}
 so the model is explosive \cite{norris:markov}. For more on explosive stochastic reaction networks, see \cite{ACKK:explosion}.  \hfill $\square$
\end{example}

We provide an example that is not strongly endotactic.  This model will be considered in Remark \ref{remark:67876}.
\begin{example}\label{ex:proper_not_transversal}
 The reaction network
 $$A\rightleftharpoons 2B, \quad A+C\rightleftharpoons B+C$$
 is not strongly endotactic. Indeed, consider the vector $w=(1,1,10)$: it is not orthogonal to the stoichiometric subspace since $\scal{w}{(-1,2,0)}\neq0$, $(-1,2,0)$ being the reaction vector of $A\to 2B$. It can be checked that the $w-$maximal complexes are $A+C$ and $B+C$, but there is no reaction $y\to y'\in\R$ with $y\in\{A+C, B+C\}$ and $\scal{w}{y'-y}<0$.
 
 It is interesting to note that within every stoichiometric compatibility class the amount of molecules of $C$ is kept constant, hence the above network equipped with mass-action kinetics is equivalent to
 $$B\rightleftharpoons A\rightleftharpoons 2B,$$
 for a suitable choice of rate constants. Somewhat surprisingly, the latter is strongly endotactic by Proposition \ref{prop:WR_SE}.  \hfill $\square$
\end{example}

\section{Tiers}
\label{sec:tiers}

This section is broken into 3 subsections.  In subsection \ref{sec:4def}, we introduce the relevant definitions related to tiers.  We also provide a few results related to these definitions.  In subsection \ref{sec:SE_tiers}, we provide Theorem \ref{thm:SE_tiers}, which is our main technical result and characterizes strongly endotactic networks in terms of their tier structures.  Finally, in subsection \ref{sec:lyapunov}, we collect results relating tier sequences with a commonly used Lyapunov function that plays a role in each of the subsequent results of the present paper.

\subsection{Definitions}
\label{sec:4def}
\begin{definition}
 A sequence $(x_n)_{n=0}^\infty$ of positive vectors of $\RR^d_{>0}$ is called a \emph{tier sequence} if
 $$\lim_{n\to\infty}\|\ln(x_n)\|_{\infty}=\infty$$
 and for all pairs of complexes $y,y'\in\C$ the limit
 $$\lim_{n\to\infty} x_n^{y'-y}$$
 exists (it could be infinity). Moreover, a tier sequence is \emph{proper} if for all $n, m\in\ZZ_{\geq0}$ we have $x_n-x_m\in S$.
\end{definition}

\begin{remark}\label{remark1}
 Note that, given a sequence $(x_n)_{n=0}^\infty$ of positive vectors in $\RR^d_{>0}$ with $\lim_{n\to\infty}\|\ln(x_n)\|_{\infty}=\infty$, it is always possible to extract a subsequence that is a tier sequence. This follows from the fact that there are finitely many complexes.
\end{remark}

\begin{remark}\label{remark:proper}
 The definition of tier sequence is tied to the choice of mass action kinetics for the reaction network. Indeed, $x_n^y$ is proportional to the deterministic mass action rate function associated with a reaction whose source is $y$, and $x_n^{y-y'}$ is nothing but the ratio $x_n^y/x_n^{y'}$. Hence, a sequence is a tier sequence if a ranking of the reaction rates $\ld$ along $x_n$ can be made, in the sense specified by the next definition. \reply{We also note here that the focus of this paper will be on the relative behavior of reaction rate functions along sequences diverging to infinity, with the aim to understand the behavior of the dynamics of the associated reaction network. Since the stoichiometric compatibility classes are invariant sets for both the deterministic and the stochastic models, it makes sense to restrict our analysis to proper tier sequences.}
\end{remark}

\begin{definition}\label{def:tier}
 Given a tier sequence $(x_n)_{n=0}^\infty$, we define tiers as subsets of $\C$ in the following recursive manner:
 \begin{enumerate}
  \item we say that a complex $y$ is in \emph{tier 1} (and write $y\in T^1_{(x_n)}$) if for all complexes $y'\in \C$
  $$\lim_{n\to\infty}x_n^{y-y'}>0;$$
  \item we say that a complex $y$ is in \emph{tier $i$} (and write $y\in T^i_{(x_n)}$) if there exists $y'\in T^{i-1}_{(x_n)}$ with
  $$\lim_{n\to\infty}x_n^{y-y'}=0$$
  and for all complexes $y'\notin\bigcup_{j=1}^{i-1}T^j_{(x_n)}$ we have
  $$\lim_{n\to\infty}x_n^{y-y'}>0.$$
 \end{enumerate}
\end{definition}

Given a tier sequence, tiers describe a partition of $\C$. We further define an order relation on $\C$ in the following way: we write $y\tle y'$ if $y\in T^i_{(x_n)}$, $y'\in T^j_{(x_n)}$ and $i\geq j$. Similarly, we write $y\tl y'$ if $y\in T^i_{(x_n)}$, $y'\in T^j_{(x_n)}$ and $i> j$. Note that the inequality on the indexes of the tiers is reversed, and $y\tl y'$ if and only if the ratio $x_n^y/x_n^{y'}$ converges to 0 as $n$ tends to infinity, meaning that $x_n^{y}$ is much smaller than $x_n^{y'}$ for large $n$. Finally, we write $y\te y'$ if $y$ and $y'$ are in the same tier. Note that by definition for all complexes $y\in\C$ we have $y\te y$.

\reply{
\begin{definition}
 We say that a tier sequence $(x_n)_{n=0}^\infty$ is \emph{transversal} if there exists at least one reaction $y\to y'\in\R$ such that
 $$\lim_{n\to\infty}|\ln(x_n^{y'-y})|=\infty.$$
\end{definition}
\begin{remark}
 Equivalently, transversal tier sequences can be defined as those tier sequences for which there exists $y\to y'\in\R$ such that the complexes $y$ and $y'$ are in different tiers. These reactions will play an important role in the proofs of our results. We will prove in Lemma~\ref{lem:proper_transversal} that all proper tier sequences (which are those we will focus on, see Remark~\ref{remark:proper}) are transversal.
\end{remark}
}

\begin{example}\label{ex:tier_sequences}
 Consider the reaction network
 $$A\rightleftharpoons B \rightleftharpoons 2C$$
 and the sequences $(x_n)_{n=0}^\infty$ and $(\hat x_n)_{n=0}^\infty$ defined by
 $$x_n=\left(\frac{1}{n}, 5-\frac{1}{n}-\frac{1}{2\sqrt{n}}, \frac{1}{\sqrt{n}}\right)\quad\text{and}\quad 
 \hat x_n=\left(e^n,2e^n,\frac{1}{n}\right).$$
Then, $(x_n)_{n=0}^\infty$ is a proper tier sequence, which we demonstrate now.  The entries $x_{n,1}$ and $x_{n,3}$ go to zero as $n$ goes to infinity, which implies $\lim_{n\to\infty}\|\ln(x_n)\|_{\infty}=\infty$. Moreover, 
 $$\lim_{n\to\infty}x_n^{(-1,0,2)}=1\quad\text{and}\quad\lim_{n\to\infty}x_n^{(-1,1,0)}=\infty,$$
 which implies that $(x_n)_{n=0}^\infty$ is a tier sequence and $A\te 2C$ and $A\tl B$. Finally, $(x_n)_{n=0}^\infty$ is proper because for any $n\geq 1$
 $$x_{n+1}-x_n=\left(\frac{1}{n+1}-\frac{1}{n}\right)(1,-1,0)+\left(\frac{1}{2\sqrt{n+1}}-\frac{1}{2\sqrt{n}}\right)(0, -1, 2)\in S.$$
 For what concerns $(\hat x_n)_{n=0}^\infty$, we still have $\lim_{n\to\infty}\|\ln(\hat x_n)\|_{\infty}=\infty$. Moreover, 
 $$\lim_{n\to\infty}\hat x_n^{(0,-1,2)}=0\quad\text{and}\quad\lim_{n\to\infty}x_n^{(-1,1,0)}=2,$$ 
 so $(\hat x_n)_{n=0}^\infty$ is a tier sequence and $A\tewrt{\hat x_n} B$ and $2C\tlwrt{\hat x_n} A$. Finally, $(\hat x_n)_{n=0}^\infty$ is transversal but not proper, indeed
 $$\lim_{n\to\infty}|\ln(\hat x_n^{(0,-1,2)})|=\infty$$
 but for any $n\geq 1$
 $$\scal{\hat x_{n+1}-\hat x_{n}}{(2,-2,1)}=-2(e^{n+1}-e^n)+\frac{1}{n+1}-\frac{1}{n}\neq0,$$
 and $(2,-2,1)$ is orthogonal to $S$ (hence $\hat x_{n+1}-\hat x_{n}\notin S$).\hfill $\square$
\end{example}

The following result connects proper and transversal tier sequences. As illustrated in Example \ref{ex:tier_sequences}, the converse does not hold.

\begin{lemma}\label{lem:proper_transversal}
 A proper tier sequence is transversal.
\end{lemma}
\begin{proof}
 Consider a proper tier sequence $(x_n)_{n=0}^\infty$. By definition, 
 $$\lim_{n\to\infty}\|\ln(x_n)\|_\infty=\infty$$
 and
 $$\lim_{n\to\infty}|\ln(x_n^{y'-y})|$$
 exists for any $y\to y'\in\R$. After potentially considering a subsequence, we may assume that for any $n\geq0$ 
 \begin{align*}
  x_{n+1,i}\geq x_{n,i}&\quad\text{if }\limsup_{n\to\infty}\ln(x_{n,i})=\infty;\\
  x_{n+1,i}\leq x_{n,i}&\quad\text{if }\liminf_{n\to\infty}\ln(x_{n,i})=-\infty,
 \end{align*}
 which implies that the above $\limsup$ and $\liminf$ are limits. It also follows that
$$\lim_{n\to\infty}|\ln(x_{n,i})|=\infty$$
 for at least one index $1\leq i\leq d$. Hence, by \cite[Theorem 3.9]{A:boundedness} there exists a vector $w\in\RR^d$ such that
 \begin{align*}
  w_i>0&\quad\text{if and only if } \lim_{n\to\infty}\ln(x_{n,i})=\infty;\\
  w_i<0&\quad\text{if and only if } \lim_{n\to\infty}\ln(x_{n,i})=-\infty;\\
  \scal{w}{y'-y}=0&\quad\text{if } y\te y'.
 \end{align*}
\reply{Note that $w_i = 0$ if $\limsup_{n\to \infty} |\ln(x_{n,i})| < \infty$.} In particular, it follows that 
 $$\lim_{n\to\infty} \scal{w}{x_n}=\begin{cases}
                                   \infty&\text{if }\lim_{n\to\infty} \|x_n\|_\infty=\infty;\\
                                   0&\text{otherwise}
                                  \end{cases}.$$
We will show that there must be an $\hat n \ge 1$ for which $\scal{w}{x_{\hat n}} \neq 0$.  First, if $\lim_{n\to \infty} \scal{w}{x_n} = \infty$, the assertion is clear.  If, on the other hand, $\lim_{n\to \infty} \scal{w}{x_n} = 0$, then none of the $x_{n,i}$ converge to infinity.  Since all the vectors $\{x_n\}_{n=0}^\infty$ are positive, and at least one of $x_{n,i}$ converges to zero, we may conclude that $\scal{w}{x_{n}}< 0$ for all $n$.

 If $(x_n)_{n=0}^\infty$ were not transversal, then we would have
 $$\lim_{n\to\infty}|\ln(x_n^{y'-y})|<\infty$$
 for any reaction $y\to y'\in\R$, which would imply that $y\te y'$ for any $y\to y'\in\R$. It would follow that $\scal{w}{y'-y}=0$ for any $y\to y'\in\R$, which means $w\in S^\perp$. Let $\hat n\geq 1$ be such that $\scal{w}{x_{\hat n}}\neq 0$. Since $(x_n)_{n=0}^\infty$ is proper, we have
 $$\lim_{n\to\infty} \scal{w}{x_n}=\scal{w}{x_{\hat n}}+\lim_{n\to\infty} \scal{w}{x_n-x_{\hat n}}=\scal{w}{x_{\hat n}}\notin\{0,\infty\}.$$
 This is a contradiction, and the proof is concluded.
\end{proof}

 For notational convenience, we give the following definition.

\begin{definition}
 Define $\C^S\subseteq\C$ to be the set of source complexes. Given a tier sequence $(x_n)_{n=0}^\infty$, we define \emph{source tier 1} to be the set
 $$T^{1,S}_{(x_n)}=\{y\in\C^S\,:\, y'\tle y\text{ for all }y'\in\C^S\}.$$
\end{definition}

 The following is a key concept of this paper, and will provide a characterization of strongly endotactic networks.

\begin{definition}\label{def:tier_descending}
 We say that a tier sequence $(x_n)_{n=0}^\infty$ is \emph{tier descending} if both the following statements hold:
 \begin{enumerate}
  \item for all $y\in T^{1,S}_{(x_n)}$ and all $y\to y'\in \R$ we have $y'\tle y$;
  \item there exist $y\in T^{1,S}_{(x_n)}$ and $y\to y'\in \R$ with $y'\tl y$. 
 \end{enumerate}
 Moreover, we say that a reaction network $\G$ is \emph{tier descending} if  all transversal tier sequences are tier descending.
\end{definition}

\subsection{Relation between strongly endotactic networks and its tiers}
\label{sec:SE_tiers}

We now state our first main result, which provides a characterization of strongly endotactic networks in terms of  tiers.

\begin{theorem}\label{thm:SE_tiers}
 A reaction network is strongly endotactic if and only if it is tier descending.
\end{theorem}

Before proceeding with the proof of  Theorem \ref{thm:SE_tiers}, we present an immediate corollary.

\begin{corollary}\label{cor:proper}
 If a reaction network is strongly endotactic, then every proper tier sequence is tier descending. Moreover, if $S=\RR^d$ then a reaction network is strongly endotactic if and only if every proper tier sequence is tier descending.
\end{corollary}
\begin{proof}
 The first part of the result follows from Lemma \ref{lem:proper_transversal} and Theorem \ref{thm:SE_tiers}. Moreover, if $S=\RR^d$ then any transversal tier sequence is proper (since all sequences are proper in this case), and the proof follows from Theorem \ref{thm:SE_tiers}.
\end{proof}
\begin{remark}\label{remark:67876}
It is tempting to believe that if every proper tier sequence of a reaction network is tier descending, then the network is strongly endotactic.  By Corollary \ref{cor:proper} we see that this is true in the case when $S = \RR^d$.  However, this statement is false, in general.  
 As an example, consider the reaction network
 $$A\rightleftharpoons 2B, \quad A+C\rightleftharpoons B+C.$$
 The network is not strongly endotactic, as shown in Example \ref{ex:proper_not_transversal}. Nevertheless, every proper tier sequence is tier descending: since no reaction changes the amount of molecules of the species $C$, every proper tier sequence $(x_n)_{n=0}^\infty$ is of the form
 $$x_n=(x_{n,1}, x_{n,2}, c)$$
 for a constant $c\in\RR_{>0}$. It is then easy to check that  $(x_n)_{n=0}^\infty$ is tier descending if and only if $(\hat x_n)_{n=0}^\infty$ defined by
 $$\hat x_n=(x_{n,1}, x_{n,2})$$ 
 is tier descending for
 $$B\rightleftharpoons A\rightleftharpoons 2B.$$
 The latter is strongly endotactic by Proposition \ref{prop:WR_SE}. Hence, each proper tier sequence (such as $(\hat x_n)_{n=0}^\infty$) is tier descending by Corollary \ref{cor:proper}, thus proving our claim. 
\end{remark}

We now proceed by providing a key lemma that will be used in the proof of Theorem \ref{thm:SE_tiers}.  

\begin{lemma}\label{lem:logarithm_decomposition}
 If $(x_n)_{n=0}^\infty$ is a tier sequence, then there exist $\ell\in\ZZ$ with $0<\ell\leq d$, sequences of positive real numbers $(m^1_n)_{n=0}^\infty$, $(m^2_n)_{n=0}^\infty$, $\dots$, $(m^\ell_n)_{n=0}^\infty$, a sequence of real vectors $(C_n)_{n=0}^\infty$, vectors $\alpha_1, \alpha_2, \dots, \alpha_\ell\in\RR^d$ and a subsequence $(x_{n_k})_{k=0}^\infty$ such that:
 \begin{enumerate}
  \item\label{part1} $\ln(x_{n_k})=\sum_{i=1}^\ell m^i_{n_k}\alpha_i + C_{n_k}$;
  \item\label{part2} $\limsup_{k\to\infty}\|C_{n_k}\|_\infty<\infty$;
  \item\label{part3} For all $1\leq i\leq\ell$ we have $\lim_{k \to\infty} m^i_{n_k}=\infty$, and if $1\leq j<i\leq \ell$ then $\lim_{k\to\infty} m^i_{n_k}/m^j_{n_k}=0$;
  \item\label{part4} if $y'\te y$ then $\scal{y'-y}{\alpha_i}=0$ for all $1\leq i\leq \ell$;
  \item\label{part5} if $y'\tl y$ then  
  \begin{equation}\label{eq:99887654}
  i_{y, y'}=\min\{1\leq i\leq \ell\,:\, \scal{\alpha_i}{y'-y}\neq 0\}
  \end{equation}
  exists and $\scal{\alpha_{i_{y, y'}}}{y'-y}<0$.
 \end{enumerate}
\end{lemma}

\begin{remark}
Parts \ref{part1} and \ref{part2} of the lemma show that the logarithm of a tier sequence can be substantially decomposed into fixed vectors, $\alpha_i$, apart from a bounded error term, $C_{n_k}$.  Part \ref{part3} then shows that if $i< j$, then the influence of the  vector $\alpha_i$ is greater than the influence of the vector $\alpha_j$.  Finally, by parts \ref{part4} and \ref{part5} we see that the $\alpha_i$'s separate complexes in a natural manner among the tiers.

As an example, consider the reaction network
$$A\rightleftharpoons B \rightleftharpoons 2C$$
and the tier sequence 
$$x_n=\left(\frac{1}{n}, 5-\frac{1}{n}-\frac{1}{2\sqrt{n}}, \frac{1}{\sqrt{n}}\right),$$
 introduced in Example \ref{ex:tier_sequences}. We have
 $$\ln(x_{n_k})=\ln(n)\left(-1,0,-\frac{1}{2}\right)+C_n,$$
 where
 $$C_n=\left(0, \ln\left(5-\frac{1}{n}-\frac{1}{2\sqrt{n}}\right),0\right).$$
 Note that $\|C_n\|_\infty<\ln(5)$ for all $n>1$. Moreover, recall that $A\te 2C$ and $A\tl B$, which is implied also by parts \ref{part4} and \ref{part5} of the lemma, since
 $$\left\langle(1,0,-2),\left(-1,0,-\frac{1}{2}\right)\right\rangle=0\quad\text{and}\quad\left\langle(1,-1,0),\left(-1,0,-\frac{1}{2}\right)\right\rangle<0.$$
\end{remark}

\begin{proof}[Proof of Lemma \ref{lem:logarithm_decomposition}]
 Define $m^1_n=\|\ln(x_n)\|_\infty$. Note that for any $n\geq0$ we have $\|\ln(x_n)/m^1_n\|_\infty=1$. Hence, we can consider a subsequence of $(x_n)_{n=0}^\infty$ such that
 \begin{equation}\label{eq:768769}
 \alpha_1=\lim_{k\to\infty}\frac{\ln(x_{n_k})}{m^1_{n_k}}
 \end{equation}
 exists. We further note that $\alpha_1$ cannot be zero since it is the limit of a sequence of points in the ball of radius 1 with respect to $\|\cdot\|_\infty$ in $\RR^d$. 
 
 Since the dimension of the vectors $x_n$ is $d<\infty$, we can further choose a subsequence such that the maximal absolute values of the entries of $\ln(x_{n_k})$ are always obtained in the same position. This implies that at least one entry of $\ln(x_{n_k})$ has absolute value constantly equal to $m^1_{n_k}$.  Moreover, by \eqref{eq:768769} the sign of such entries will stabilize for $k$ large enough.  Hence,  the vectors
 $$\ln(x_{n_k})-m^1_{n_k}\alpha_1$$
 have at least one component constantly equal to zero for $k$ large enough.  
 
 We define $m^i_{n_k}$ and $\alpha_i$ iteratively in the following way: for each $j\geq 2$, if
 $$\limsup_{k\to\infty}\|\ln(x_{n_k})-\sum_{i=1}^{j-1} m^{i}_{n_k}\alpha_i\|_\infty=\infty,$$
 then define $m^j_{n_k}=\|\ln(x_{n_k})-\sum_{i=1}^{j-1} m^{i}_{n_k}\alpha_i\|_\infty$. By potentially considering a subsequence of $(x_{n_k})_{k=0}^\infty$, we can assume that
 $$\alpha_j=\lim_{k\to\infty}\frac{\ln(x_{n_k})-\sum_{i=1}^{j-1} m^{i}_{n_k}\alpha_i}{m^j_{n_k}}$$
 exists. As before, note that $\alpha_j$ cannot be zero. Moreover, we can choose a subsequence such that the maximal absolute values of the entries of $\ln(x_{n_k})-\sum_{i=1}^{j-1} m^{i}_{n_k}\alpha_i$ are always obtained in the same position, so by induction it follows that at least $j-1$ components of 
 $\ln(x_{n_k})-\sum_{i=1}^{j-1} m^{i}_{n_k}\alpha_i$ are equal to zero for $k$ large enough (the argument is the same as for $j=1$, which serves as base case).
 In particular, it follows that there exists a number $\ell\leq d$ such that
 $$\limsup_{k\to\infty}\left\|\ln(x_{n_k})-\sum_{i=1}^\ell m^i_{n_k}\alpha_i\right\|_\infty<\infty.$$
 
 We define
 $$C_n=\ln(x_n)-\sum_{i=1}^\ell m^i_n\alpha_i.$$
 Parts \eqref{part1} and \eqref{part2} trivially hold by the definition of $C_n$. For part \eqref{part3}, note that for all $2\leq j\leq \ell$
 $$\lim_{k\to\infty} \frac{m^j_{n_k}}{m^{j-1}_{n_k}}=\lim_{k\to\infty}\left\|\frac{\ln(x_{n_k})-\sum_{i=1}^{j-1} m^{i}_{n_k}\alpha_i}{m^{j-1}_{n_k}}\right\|_\infty=
 \|\alpha_{j-1}-\alpha_{j-1}\|_\infty=0.
 $$
 
 For part \eqref{part4}, consider $y\te y'$. Then,
 $$0<\lim_{k\to\infty} x_{n_k}^{y'-y}<\infty.$$
 By taking the logarithm, it follows that 
 $$-\infty<\lim_{k\to\infty} \ln(x_{n_k}^{y'-y})<\infty.$$
 Hence, since $m^1_{n_k}$ tends to infinity as $k$ tends to infinity, we have
 $$0=\lim_{k\to\infty} \frac{\ln(x_{n_k}^{y'-y})}{m^1_{n_k}}=\scal{\alpha_1}{y'-y}.$$
 We complete the proof of part \eqref{part4} by induction: consider $1<j\leq\ell$ and assume that the statement holds for any $1\leq i\leq j-1$. Then, by part \eqref{part1} and since $m^j_{n_k}$ tends to infinity as $k$ tends to infinity, we have
 \begin{align*}
  0&=\lim_{k\to\infty} \frac{\ln(x_{n_k}^{y'-y})}{m^j_{n_k}}=\lim_{k\to\infty} \frac{\scal{\sum_{i=1}^\ell m^i_{n_k}\alpha_i + C_{n_k}}{y'-y}}{m^j_{n_k}}\\
  &=\lim_{k\to\infty} \frac{\scal{\sum_{i=j}^\ell m^i_{n_k}\alpha_i + C_{n_k}}{y'-y}}{m^j_{n_k}}=\scal{\alpha_j}{y'-y}.
 \end{align*}
 
 Finally, for part \eqref{part5} consider $y'\tl y$. Then, we have
 $$\lim_{k\to\infty} x_{n_k}^{y'-y}=0,$$
 which implies
 \begin{equation}\label{eq:54322}
 -\infty=\lim_{k\to\infty} \ln(x_{n_k}^{y'-y})=\lim_{k\to\infty}\left(\sum_{i=1}^\ell m^i_{n_k}\scal{\alpha_i}{y'-y} + \scal{C_{n_k}}{y'-y}\right).
 \end{equation}
 Since the values $\|C_{n_k}\|_\infty$ are bounded uniformly in $k$, we have
 $$\lim_{k\to\infty}\sum_{i=1}^\ell m^i_{n_k}\scal{\alpha_i}{y'-y}=-\infty,$$
 which implies that
 \begin{equation*}
 i_{y, y'}=\min\{1\leq i\leq \ell\,:\, \scal{\alpha_i}{y'-y}\neq 0\}
 \end{equation*}
 exists. 
  Moreover, by part \ref{part3} we have
 $$\scal{\alpha_{i_{y, y'}}}{y'-y}=\lim_{k\to\infty} \frac{\ln(x_{n_k}^{y'-y})}{m^{i_{y, y'}}_{n_k}}.$$
 By construction, the term on the left is non-zero.  Further, by \eqref{eq:54322} the right-hand size is non-positive.  Hence, $\scal{\alpha_{i_{y, y'}}}{y'-y} < 0$,
 which concludes the proof.
\end{proof}

Now we are able to prove Theorem \ref{thm:SE_tiers}.

\begin{proof}[Proof of Theorem \ref{thm:SE_tiers}]
 Assume that the network is tier descending. Consider a vector $w$ that is not orthogonal to the stoichiometric subspace $S$. Consider the sequence $(x_n)_{n=0}^\infty$ defined by
 $$x_n=e^{nw}.$$
 We have
 $$\lim_{n\to\infty}\|\ln(x_n)\|_\infty=\lim_{n\to\infty}n\|w\|_\infty=\infty$$
 and for any two complexes $y, y'\in\C$
 \begin{equation}\label{eq:lim_log}
\lim_{n\to\infty}\ln(x_n^{y'-y})=\lim_{n\to\infty}n\scal{w}{y'-y}=\begin{cases}
                                                                             -\infty&\text{if }\scal{w}{y'-y}<0\\
                                                                             0&\text{if }\scal{w}{y'-y}=0\\
                                                                             \infty&\text{if }\scal{w}{y'-y}>0\\
                                                                            \end{cases}  .
 \end{equation}
 Hence, $(x_n)_{n=0}^\infty$ is a tier sequence. Moreover, it is transversal: since $w$ is not orthogonal to $S$, there exists a reaction $y\to y'$ with $\scal{w}{y'-y}\neq0$, which implies $\lim_{n\to\infty}|\ln(x_n^{y'-y})|=\infty$. It follows that $(x_n)_{n=0}^\infty$ is tier descending, which together with equation \ref{eq:lim_log} concludes the proof of one direction of the result.

 For the other direction, we suppose that the network is strongly endotactic.  Let $(x_n)_{n=0}^\infty$ be a transversal tier sequence. In order to prove the result, it is sufficient to construct a vector $w$ such that
 \begin{enumerate}
  \item $w\notin S^\perp$;
  \item $\scal{w}{y'-y}=0$ if and only if $y'\te y$, and $\scal{w}{y'-y}<0$ if and only if $y'\tl y$.
 \end{enumerate}
 Indeed, if such a vector is constructed, then it follows that the set of $w-$maximal complexes coincides with $y\in T^{1,S}_{(x_n)}$, and by Definition \ref{def:strongly_endotactic} the sequence $(x_n)_{n=0}^\infty$ is tier descending. 
 
 Consider a subsequence $(x_{n_k})_{k=0}^\infty$ as in Lemma \ref{lem:logarithm_decomposition}, such that there exist $\ell\in\ZZ$ with $0<\ell\leq d$, sequences of positive real numbers $(m^1_{n_k})_{k=0}^\infty$, $(m^2_{n_k})_{k=0}^\infty$, $\dots$, $(m^\ell_{n_k})_{k=0}^\infty$, $(C_{n_k})_{k=0}^\infty$,  and vectors $\alpha_1, \alpha_2, \dots, \alpha_\ell\in\RR^d$ such that
 $$\ln(x_{n_k})=\sum_{i=1}^\ell m^i_{n_k}\alpha_i + C_{n_k}.$$
 Note that $(x_{n_k})_{k=0}^\infty$ is still a transversal tier sequence, and the tier structures of $(x_n)_{n=0}^\infty$ and of its subsequence $(x_{n_k})_{k=0}^\infty$ are identical, meaning that for any $i\geq 1$ we have $T^i_{(x_n)}=T^i_{(x_{n_k})}$.  Let
 $$w=\sum_{i=1}^\ell v_i\alpha_i,$$
 with the positive constants $v_i$ defined recursively as follows: $v_\ell=1$ and
 $$v_i=1+\max_{\substack{y\to y'\in\R\\ \scal{\alpha_i}{y'-y}\neq 0}}\left|\frac{\sum_{j=i+1}^\ell v_j\scal{\alpha_j}{y'-y}}{\scal{\alpha_i}{y'-y}}\right|\quad\text{for }1\leq i\leq \ell-1.$$
 We have the following:
 \begin{enumerate}
  \item Since $(x_{n_k})_{k=0}^\infty$ is transversal and since $\|C_{n_k}\|_\infty$ are bounded, there must exist a reaction $y\to y'$ and a vector $\alpha_i$ such that $\scal{\alpha_i}{y'-y}\neq 0$. Let
  $$\hat{\imath}=\min_{1\leq i\leq \ell\,:\, \scal{\alpha_i}{y'-y}\neq 0}.$$
  By definition of the constants $v_i$, we have
  $$|v_{\hat{\imath}}\scal{\alpha_{\hat{\imath}}}{y'-y}|>\left|\sum_{j=\hat{\imath}+1}^\ell v_j\scal{\alpha_j}{y'-y}\right|,$$
  hence
  $$\scal{w}{y'-y}=\sum_{j=\hat{\imath}}^\ell v_j\scal{\alpha_j}{y'-y}\neq0,$$
  which is equivalent to say that $w\notin S^\perp$.
  \item By Lemma \ref{lem:logarithm_decomposition}\eqref{part4}\eqref{part5}, $y'\te y$ if and only if $\scal{\alpha_i}{y'-y}=0$ for all $1\leq i\leq \ell$. By the definition of $w$   the latter is in turn equivalent to $\scal{w}{y'-y}=0$. Moreover, $y'\tl y$ if and only if $\scal{\alpha_{i_{y,y'}}}{y'-y}<0$, where $i_{y,y'}$ is defined in \eqref{eq:99887654},  
 which by definition of the constants $v_i$ is equivalent to
  $$\scal{w}{y'-y}=\sum_{j=i_{y,y'}}^\ell v_j\scal{\alpha_j}{y'-y}<0.$$

 \end{enumerate}
 The proof is then concluded. 
\end{proof}

\subsection{Tier sequences and Lyapunov functions}\label{sec:lyapunov}

Let $u(x) : \mathbb{R} \to \mathbb{R}_{\ge 0}$ be the function
\begin{align}\label{eq:genelized lyapunov}
u(x)= \begin{cases}
x(\ln{x}-1)+1  \quad &\text{if} \quad x>0,\\
1 \quad &\text{otherwise}.
\end{cases}
\end{align}
Then we define
\begin{equation}\label{eq:lyapunov}
 U(x)=1 + \sum_{i=1}^d u(x_i).
\end{equation}
This function has been utilized often as a Lyapunov function in the context of reaction network theory.  In particular, it was utilized in the foundational papers of the field in order  prove local asympotic stability of complex balanced deterministic mass action systems \cite{F1,H}.  Moreover,  it (or slight modifications thereof) has notably been used to derive the results of \cite{A:single, A:boundedness, GMS:geometric, ADE:geometric}, which are of direct interest for the present paper. More discussion on the role of Lyapunov functions for stochastic reaction networks can be found in \cite{ACGW2015} and \cite{AK2017}. 

In the present section, we will unveil some important connections between tier sequences and the Lyapunov function \eqref{eq:lyapunov} by extending the techniques of \cite{A:single} to the  setting of tier descending networks. We will then  use these connections to develop the results presented in sections \ref{sec:permanence}, \ref{sec:LDP}, and \ref{sec:PR_SE}.

\begin{lemma}\label{lem:dominated_reaction}
 Consider a tier descending reaction network $\G$ and let $(x_n)_{n=0}^\infty$ be a transversal tier sequence. Then, for any $y\to y'\in \R$ with $y\tle y'$ there exists $y^\star\in \C$ and $y^\star\to y^{\star \star}\in\R$ such that $y\tle y^\star$, $y^{\star \star}\tl y^\star$ and
  for any choice of  $c_1,c_2 \in \RR_{>0}$ and $c_3,c_4 \in \RR$ there exists $N<\infty$ with
 \begin{equation}\label{eq:dominated_reaction_negative}
  c_1x_n^{y^\star}\left(\ln(x_n^{y^{\star\star}- y^\star}) + c_3\right)+c_2x_n^{ y}\left(\ln(x_n^{y'- y})+c_4\right) < 0\quad \text{for all }n\geq N.
 \end{equation}
 Moreover, \reply{if there exists a $c\in \RR_{>0}$ for which} $x_n^{y^\star}\geq c>0$ for all $n$, then for any choice of  $c_1,c_2 \in \RR_{>0}$ and $c_3,c_4 \in \RR$ we have
 \begin{equation}\label{eq:dominated_reaction_infinity}
  \lim_{n\to\infty}\left(c_1x_n^{y^\star}\left(\ln(x_n^{y^{\star\star}- y^\star}) + c_3\right)+c_2x_n^{ y}\left(\ln(x_n^{y'- y})+c_4\right)\right)=-\infty.
 \end{equation}
\end{lemma}
\begin{proof}
 Fix $y\to y'\in \R$. We consider two cases separately: $y \te y'$ and $y\tl y'$.
 
 \vspace{.2in}
 
 \noindent \textbf{Case 1.}  Assume that $y\te y'$. Then
 $$\lim_{n\to\infty}|\ln(x_{n}^{y'-y})|<\infty.$$
 By the definition of a descending reaction network there must be at least one reaction $y^\star\to y^{\star \star}$ with $y^\star\in T^{1,S}_{(x_n)}$ (implying $y\tle y^\star$) and $y^{\star \star}\tl y^\star$. Hence, we have
 $$\lim_{n\to\infty} x_n^{y-y^\star}<\infty$$
 and
 \begin{equation*}
 \lim_{n\to\infty} \ln(x_n^{y^{\star\star}-y^\star})=-\infty.
 \end{equation*}
 It follows that
 \begin{align*}
 c_1x_n^{y^\star}\big(\ln(x_n^{y^{\star\star}- y^\star})& + c_3\big)+c_2x_n^{ y}\big(\ln(x_n^{y'- y})+c_4\big)\\
 &=x_n^{y^\star}\left(c_1\left(\ln(x_n^{y^{\star\star}- y^\star}) + c_3\right)+c_2x_n^{ y-y^\star}\left(\ln(x_n^{y'- y})+c_4\right)\right)
 \end{align*}
 is negative for $n$ large enough, which proves \eqref{eq:dominated_reaction_negative}. Moreover, if $x_n^{y^\star}\geq c>0$, then   \eqref{eq:dominated_reaction_infinity} follows.

\vspace{.2in}

\noindent \textbf{Case 2.} \reply{We prove the result by contradiction.} Assume that  $y\tl y'$. If \eqref{eq:dominated_reaction_negative} did not hold, then there would exist a subsequence $(x_{n_k})_{k=0}^\infty$ such that for any $y^\star\to y^{\star \star}\in\R$ with $y\tle y^\star$ and $y^{\star \star}\tl y^\star$, there exist  $c_1,c_2 \in \RR_{>0}$ and $c_3,c_4 \in \RR$ with
 \begin{equation}\label{eq:not_true}
  c_1x_n^{y^\star}\left(\ln(x_n^{y^{\star\star}- y^\star}) + c_3\right)+c_2x_n^{ y}\left(\ln(x_n^{y'- y})+c_4\right)\geq 0\quad\text{for all }k\in\Z_{\geq0}.
 \end{equation}
 Our aim is to prove that such a subsequence does not exist.
 
 Every subsequence of a descending tier sequence is still a descending tier sequence. Hence, by potentially considering a further subsequence, we can assume that $(x_{n_k})_{k=0}^\infty$ is as in Lemma \ref{lem:logarithm_decomposition}.

 Consider the sequence $(\tilde{x}_{n_k})_{k=0}^\infty$ defined by
 \begin{equation}\label{eq:000098765678}
 \ln(\tilde{x}_{n_k})= \sum_{i=1}^{i_{y',y}}m^i_{n_k}\alpha_i
 \end{equation}
 where $i_{y',y}$ is as defined in \eqref{eq:99887654}, and exists by Lemma \ref{lem:logarithm_decomposition}\eqref{part5}.  We will first show that 
 $(\tilde{x}_{n_k})_{k=0}^\infty$ is also a transversal tier sequence, and is therefore tier descending. \reply{We then prove that there exist $y^\star\in \C$ and $y^\star\to y^{\star \star}\in\R$ such that $y\tlwrt{\tilde x_{n_k}} y^\star$, $y^{\star \star}\tlwrt{\tilde x_{n_k}} y^\star$. Finally, we will prove that these complex orderings are valid also when considerering the original tier sequence $(x_n)_{n=0}^\infty$, as required by the lemma.}
 
 By Lemma \ref{lem:logarithm_decomposition}\eqref{part3}, we have
 \begin{equation*}
 \lim_{k\to \infty} \frac{\left\|\ln(\tilde{x}_{n_k})\right\|_\infty}{m_{n_k}^{1} \left\|\alpha_{1}\right\|_\infty} = 1,
 \end{equation*}
 and so
  $\lim_{k\to\infty}\|\ln(\tilde{x}_{n_k})\|_\infty=\infty$. Furthermore,  for any two complexes $\tilde{y}, \tilde{y}'\in \C$ the limit
 $$\lim_{k\to \infty} \tilde{x}_{n_k}^{\tilde{y}'-\tilde{y}}=\lim_{k\to \infty} e^{\sum_{i=1}^{i_{y',y}}m^i_{n_k}\scal{\alpha_i}{\tilde{y}'-\tilde{y}}}$$
 exists (it can potentially be infinity). Hence, $(\tilde x_{n_k})_{k=0}^\infty$ is a tier sequence. Moreover,
 $$\lim_{k\to \infty} |\ln(\tilde{x}_{n_k}^{y'-y})|=\lim_{k\to \infty} \left |  \sum_{i = 1}^{i_{y',y}} m_{n_k}^i \scal{\alpha_i}{y'-y} \right| = \lim_{k\to \infty} m^{i_{y',y}}_{n_k}|\scal{\alpha_{i_{y',y}}}{y'-y}|=\infty.$$
 Hence, $(\tilde{x}_{n_k})_{k=0}^\infty$ is a transversal tier sequence.  Combining this with the fact that $\G$ is a tier descending reaction network, we may conclude that  $(\tilde{x}_{n_k})_{k=0}^\infty$ is tier descending. 
 Since $(\tilde{x}_{n_k})_{k=0}^\infty$ is a tier sequence, Lemma  \ref{lem:logarithm_decomposition} guarantees that it can be decomposed as detailed therein.  It is straightforward to prove that the vectors and coefficients as  constructed in the proof of the lemma coincide with the $m_{n_k}^i$ and $\alpha_i$  in \eqref{eq:000098765678}, for $1 \le i \le i_{y',y}$.

 By Lemma \ref{lem:logarithm_decomposition}\eqref{part3}\eqref{part5} we have
 $$\lim_{k\to \infty} \ln(\tilde{x}_{n_k}^{y'-y})=\lim_{k\to \infty} m^{i_{y',y}}_{n_k}\scal{\alpha_{i_{y',y}}}{y'-y}=-\infty,$$
 allowing us to conclude that $\lim_{k\to \infty} \tilde{x}_{n_k}^{y'-y}=0$.  Thus,  $y\tlwrt{\tilde x_{n_k}} y'$. Since $(\tilde{x}_{n_k})_{k=0}^\infty$ is tier descending, $y$ cannot be in $T^{1,S}_{(\tilde x_{n_k})}$.   Hence, there must exist a complex $y^\star$ with $y \tlwrt{\tilde x_{n_k}} y^\star$ and a reaction $y^\star\to y^{\star \star}\in \R$ with $y^{\star \star}\tlwrt{\tilde x_{n_k}} y^\star$. Combining $y \tlwrt{\tilde x_{n_k}} y^\star$ with  Lemma \ref{lem:logarithm_decomposition}\eqref{part5}, it follows that $i_{y^\star,y} \leq i_{y',y}$.  Hence, by Lemma \ref{lem:logarithm_decomposition}\eqref{part3} we  may conclude
 $$\lim_{k\to\infty}\ln(\tilde{x}_{n_k}^{{y}-y^\star})=\lim_{k\to\infty}\ln(x_{n_k}^{{y}-y^\star}),$$
 as they are both asymptotically equivalent to the same term.  Therefore, the latter is negative infinity and $y\tlwrt{x_{n_k}} y^\star$.
 
 Similarly as above, since  $y^{\star \star} \tlwrt{\tilde x_{n_k}} y^\star$ we may conclude that $i_{y^{\star \star}, y^\star} \leq i_{y',y}$ and  $y^{\star \star} \tlwrt{x_{n_k}} y^\star$.

 \reply{Now we prove \eqref{eq:dominated_reaction_negative}.} Combining $x_{n_k}^{y^\star}>0$ and $y^{\star \star} \tlwrt{x_{n_k}} y^\star$ we know that for $k$ large enough
 \begin{equation}\label{eq:negative_contribution}
  x_{n_k}^{y^\star}\left(\ln(x_{n_k}^{y^{\star \star}-y^\star})+c_3\right)< 0.
 \end{equation}
 Moreover, combining $y \tlwrt{x_{n_k}} y^\star$,  $i_{y^\star, y^{\star \star}} \leq i_{y',y}$, and Lemma \ref{lem:logarithm_decomposition}\eqref{part3}\eqref{part5} we have
 \begin{equation}\label{eq:stronger}
  \lim_{k\to\infty}\frac{x_{n_k}^{y^\star}\left(\ln(x_{n_k}^{y^{\star \star}-y^\star})+c_3\right)}{x_{n_k}^{y}\left(\ln(x_{n_k}^{y'-y})+c_4\right)}=\lim_{k\to\infty}x_{n_k}^{y^\star-y}\frac{m^{i_{y^{\star\star},y^\star}}_{n_k}\scal{\alpha_{i_{y^{\star\star},y^\star}}}{y^{\star \star}-y^\star}}{m^{i_{y',y}}_{n_k}\scal{\alpha_{i_{y',y}}}{y'-y}}=-\infty,
 \end{equation}
 where we use that $\scal{\alpha_{i_{y^{\star\star},y^\star}}}{y^{\star \star} - y^\star}<0$ and $\scal{\alpha_{i_{y',y}}}{y'-y}>0$.
 By \eqref{eq:negative_contribution} and \eqref{eq:stronger}, for any positive constants $c_1, c_2$ we have
 $$\limsup_{k\to\infty} \Big(c_1x_n^{y^\star}\left(\ln(x_n^{y^{\star\star}- y^\star}) + c_3\right)+c_2x_n^{ y}\left(\ln(x_n^{y'- y})+c_4\right)\Big)< 0,$$
 which is a contradiction of \eqref{eq:not_true}, hence \eqref{eq:dominated_reaction_negative} holds.
 
 \reply{In order to prove the last part of the result,} assume that $x_n^{y^\star}\ge c > 0$, \reply{where $c$ is as in the statement of the lemma}. Let $d_1, d_2 \in \RR_{>0}$ and $d_3,d_4\in \RR$.  We must show that for the particular choice of sequence $(x_n)_{n=0}^\infty$, and the particular choice of $y^\star$ and $y^{\star \star}$ we have that 
\begin{equation}\label{978689697869876}
\lim_{n\to \infty}\left(d_1x_n^{y^\star}\left( \ln(x_n^{y^{\star\star}- y^\star}) + d_3\right)+d_2x_n^{ y}\left(\ln(x_n^{y'- y})+ d_4\right)\right)=-\infty.
\end{equation}
We may apply \eqref{eq:dominated_reaction_negative} with $c_1=d_1/2$, $c_2 = d_2$, $c_3=d_3$ and $c_4=d_4$ to conclude that for $n$ large enough we have
\begin{align}
\label{908679786867643}
\begin{split}
d_1x_n^{y^\star}&\left(\ln(x_n^{y^{\star\star}- y^\star})+d_3\right)
+d_2x_n^{ y}\left(\ln(x_n^{y'- y})+d_4\right) \\
&<
d_1x_n^{y^\star}\left(\ln(x_n^{y^{\star\star}- y^\star})+d_3\right)+
d_2\left( \frac{d_1/2}{d_2} \left|x_n^{ y^\star}\left(\ln(x_n^{y^{\star\star}- y^{\star}})+d_3\right)\right|\right)\\
&= \frac{d_1}{2} x_{n}^{y^\star} \left(\ln(x_{n}^{y^{\star\star} - y^{\star}})+d_3\right),
\end{split}
\end{align}
where we are using that $x_n^{ y}\ln(x_n^{y'- y}) > 0$ and $x_{n}^{y^\star} \ln(x_{n}^{y^{\star\star} - y^{\star}}) < 0$.
Then, since $y^{\star \star}\tl y^\star$, by Lemma \ref{lem:logarithm_decomposition}\eqref{part3}\eqref{part5} we have
 $$\lim_{n\to\infty}\ln(x_{n}^{y^{\star \star}-y^\star})=\lim_{n\to\infty}\sum_{i=i_{y^\star, y^{\star \star}}}^\ell m^i_{n}\scal{\alpha_i}{y^{\star \star}-y^\star}=-\infty. $$
 It follows that
 \begin{equation}\label{finalequation}
 \lim_{n\to\infty}\frac{d_1}{2}x_{n}^{y^\star}\left(\ln(x_{n}^{y^{\star \star}-y^\star})+d_3\right)\leq \lim_{n\to\infty}\frac{d_1}{2}c \left(\ln(x_{n}^{y^{\star \star}-y^\star})+d_3\right)=-\infty.
 \end{equation}
 Combining \eqref{finalequation} and \eqref{908679786867643} yields \eqref{978689697869876}, and completes the proof.

\end{proof}

\begin{proposition}\label{prop:lyapunov}
 Consider a tier descending reaction network $\G$. Then, for any transversal tier sequence $(x_n)_{n=0}^\infty$ and any choice of positive constants $\kappa_{y\to y'}$, there exists $N<\infty$ such that
 \begin{equation}\label{eq:lyapunov_negative}
  \sum_{y\to y'\in\R} \kappa_{y\to y'}x_n^y\ln(x_n^{y'-y})<0\quad\text{for all }n\geq N.
 \end{equation}
 Moreover, if the complex 0 is a source complex, then
 \begin{equation}\label{eq:lyapunov_infinite}
  \lim_{n\to\infty}\sum_{y\to y'\in\R} \kappa_{y\to y'}x_n^y\ln(x_n^{y'-y})=-\infty.
 \end{equation}
\end{proposition}

\begin{proof}
The result follows from noting that for any reaction $y\to y'\in \R$ either $y'\tl y$ and
$$x_n^y\ln(x_{n_k}^{y'-y})<0,$$
or $y\tle y'$ and Lemma \ref{lem:dominated_reaction} holds. Hence, since there are finitely many reactions, for any choice of positive constants $\kappa_{y\to y'}$ there exists $N<\infty$ such that \eqref{eq:lyapunov_negative} holds.

For the second part of the statement, assume that 0 is a source complex. Then, by definition of $T^{1,S}_{(x_n)}$ we have $0\tle y$ for all $y\in T^{1,S}_{(x_n)}$, which implies that for all $y\in T^{1,S}_{(x_n)}$
$$\lim_{n\to\infty} x_n^y=\lim_{n\to\infty} x_n^{y-0}>0.$$
Since $(x_n)_{n=0}^\infty$ is transversal and $\G$ is tier descending, $(x_n)_{n=0}^\infty$ is tier descending. Hence, there is a reaction $y\to y'\in\R$ with $y\in T^{1,S}_{(x_n)}$ and $y'\tl y$. Hence
$$\lim_{n\to\infty} x_n^y\ln(x_n^{y'-y})=-\infty,$$
and similarly as before \eqref{eq:lyapunov_infinite} follows from Lemma \ref{lem:dominated_reaction}.
 \end{proof}

\section{Persistence and Permanence}
\label{sec:permanence}

The paper \cite{GMS:geometric} deals with persistence and permanence of deterministic mass action systems associated with a strongly endotactic reaction network. The relevant definitions are as follows.

\begin{definition}\label{def:persistence}
 A deterministic reaction system is \emph{persistent} if for any initial condition $z(0)\in \RR^d_{>0}$
 $$\inf_{t\geq 0}z_i(t)>0\quad\text{for all }1\leq i\leq d.$$
\end{definition}

\begin{definition}\label{def:permanence}
 A deterministic reaction system is \emph{permanent} if for every set
 $$S_y=(y+S)\cap \RR^d_{>0}$$
 with $y\in\RR^d_{>0}$, there exists a compact set $K\subset S_y$ such that for any initial condition $z(0)\in S_y$
 $$\inf\{t\geq 0\,:\,z(s)\in K\quad\text{for all }s\geq t\}<\infty.$$ 
\end{definition}

Thus, a deterministic reaction system is permament if there exists a compact set in the interior of each  positive stoichiometric compatibility class that eventually attracts all the solutions with a positive initial condition in that stoichiometric compatibility class. Note that if a reaction network is permanent, then it is \reply{non-explosive and} persistent.

The following is an important result in \cite{GMS:geometric}.  It is used to prove persistence and permanence of strongly endotactic networks. In our setting, it can be derived as a corollary of the results on tier sequences stated in Section \ref{sec:lyapunov}.


\begin{corollary}\label{cor:lyapunov}
 Let $\G$ be a strongly endotactic reaction network and consider a generalization of mass action kinetics with parameter dependent and time variable rate constants:
 $$\lambda_{y\to y'}(x,t, \theta)=\kappa_{y\to y'}(t, \theta) x^y,$$
 where $\theta$ is in some parameter space $\Omega$ and $t\in\RR_{\geq0}$. Assume that there exists $\delta>0$ such that
 \begin{equation}\label{eq:log_bounded}
  \delta<\kappa_{y\to y'}(t, \theta)<\frac{1}{\delta}\quad\text{for all }t\geq0, \theta\in\Omega, y\to y'\in \R.
 \end{equation}
 Let  
 $$z(t,\theta)=z(0,\theta)+\sum_{y\to y'\in\R} (y'-y)\int_0^t\lambda_{y\to y'}(z(s,\theta),s,\theta)\,ds.$$ 
 Fix a set $S_y$ as in Definition \ref{def:permanence}. Then, there exists a compact set $\Gamma\subset S_y$ such that
 $$\frac{d}{dt}U(z(t,\theta))<0 \quad\text{if }z(t,\theta)\notin \Gamma,$$
 given that $z(0,\theta)\in S_y$ and $U(\cdot)$ is as in \eqref{eq:lyapunov}.
 In particular, it follows that for any open set $B$ containing the origin,
 $$\inf_{z(0,\theta)\in \RR^d_{>0}\setminus B}\inf_{t\geq 0}\|z(t,\theta)\|_\infty>0.$$
\end{corollary}
\begin{proof}
 If the result were not true, there would be a sequence of vectors $(x_n)_{n=0}^\infty$ in $S_y$ for which $\lim_{n\to\infty} \|\ln(x_n)\|_\infty=\infty$ and an increasing sequence of times $(t_n)_{n=0}^\infty$ such that
 \begin{equation}\label{eq:contradiction}
  \sum_{y\to y'\in\R} \kappa_{y\to y'}(t_n)x_n^y\scal{y'-y}{\nabla U(x_n)}=\sum_{y\to y'\in\R} \kappa_{y\to y'}(t_n)x_n^y\ln(x_n^{y'-y})\geq0,
 \end{equation}
 for all $n \ge 0$.
 However, by Remark \ref{remark1} we can extract a proper tier sequence from $(x_n)_{n=0}^\infty$, hence \eqref{eq:contradiction} cannot hold by \eqref{eq:log_bounded}, Proposition \ref{prop:lyapunov} and Lemma \ref{lem:proper_transversal}.
 
 The second part of the result follows by noting that the origin is a local maximum for the function $U(\cdot)$, and it is not contained in the compact set $\Gamma$. Hence, for any open set $B$ (relative to $\RR^d_{\geq0}$) that contains the origin, there exists a neighborhood $B'\subseteq B$ of the origin (relative to $\RR^d_{\geq0}$) that does not intersect $\Gamma$ (implying that $U(z(t,\theta))$ decreases if $z(t,\theta)$ is in $B'$), and such that
 $$0=\argmax_{x\in B'}U(x).$$
 Hence, there exists an open set $B''$ (relative to $\RR^d_{\geq0}$) that contains the origin and cannot be reached by any trajectory with $z(0)\in\RR^d_{>0}\setminus B$.
\end{proof}

We will also need the following results. 

\begin{lemma}\label{lem:projection}
 Assume that $\G$ is strongly endotactic, and consider a non-empty subset of species $\tilde{\Sp}\subseteq\Sp$. Let $p$ be the projection from $\RR^d$ onto the coordinates relative to the species in $\tilde{\Sp}$. Then, the reaction network $\tilde{\G}=(\tilde{\Sp}, \tilde{\C}, \tilde{\R})$ with
 \begin{align*}
  \tilde{\C}&=\{p(y)\,:\,y\in \C\}\\
  \tilde{\R}&=\{p(y)\to p(y')\,:\, y\to y'\in\R\quad\text{and}\quad p(y)\neq p(y')\}
 \end{align*}
 is strongly endotactic.
\end{lemma}
\begin{proof}
 Let $\tilde d$ be the cardinality of $\tilde{\Sp}$. For convenience, assume without loss of generality that the species of $\tilde{\Sp}$ are ordered as the first $\tilde d$ species, such that for any $x\in\RR^d_{\geq0}$ we can write $x=(\tilde x, \hat x)$ with $p(x)=\tilde x$ and $\hat x\in\RR^{d-\tilde d}_{\geq0}$. Let $(\tilde x_n)_{n=0}^\infty$ be a transversal tier sequence of $\tilde\G$, fix $\hat{x}\in\RR^{d-\tilde d}_{>0}$ and for any $n\geq0$ let $x_n=(\tilde x_n, \hat x)$. Note that for any $n\geq0$ and for any complex $y\in\C$,  $x_n^y$ is equal to $\tilde x_n^{p(y)}$ times a multiplicative constant that is independent of $n$. It follows that $(x_n)_{n=0}^\infty$ is a transversal tier sequence of $\G$, and that for any $i\geq1$ we have $y\in T^i_{(x_n)}$ if and only if $p(y)\in T^i_{(\tilde x_n)}$, which in turn implies that $(x_n)_{n=0}^\infty$ is tier descending if and only if $(\tilde x_n)_{n=0}^\infty$ is tier descending. Hence, we conclude that $\tilde \G$ is strongly endotactic by the fact that $\G$ is strongly endotactic and by Theorem \ref{thm:SE_tiers}.
\end{proof}

\begin{lemma}\label{lem:away_from_boundary}
 Consider a deterministic mass action system $(\G, \Lambda)$, and assume $\G$ is strongly endotactic. Then, for any compact set $\Upsilon\subset \RR^d_{>0}$
 \begin{gather}
  \sup_{z(0)\in\Upsilon}\sup_{t\geq0}\|z(t)\|_\infty<\infty\label{eq:unif_upper_bound}\\
  \inf_{z(0)\in\Upsilon}\inf_{t\geq0}z_i(t)>0\quad\text{for all }1\leq i\leq d.\label{eq:unif_lower_bound}
 \end{gather}
\end{lemma}
\begin{proof}
 By Corollary \ref{cor:lyapunov}, 
 we have that there exists a compact set $\Gamma\subset \RR^d_{>0}$ such that the function $U(z(t))$ is decreasing whenever $z(t)\notin \Gamma$. It follows that
 $$\sup_{z(0)\in\Upsilon}\sup_{t\geq0}U(z(t))\leq \max\left\{\max_{z(0)\in\Upsilon}U(z(0)), \max_{x\in\Gamma}U(x)\right\}<\infty,$$
 and the sets of the form $\{x\in\RR^d_{\geq0}:U(x)\leq M\}$ are compact for any finite constant $M$. So \eqref{eq:unif_upper_bound} is proven.
 
 Now, assume \eqref{eq:unif_lower_bound} does not hold: this implies that there exists $1\leq i\leq d$ such that
 $$\inf_{z(0)\in\Upsilon}\inf_{t\geq 0}z_i(t)=0.$$
 For simplicity, in the rest of the proof we will denote by $\theta$ an element of $\Upsilon$ and by $z_\theta(\cdot)$ the solution with $z_\theta(0)=\theta$. Since
 $$\{z_\theta(t)\,:\,\theta\in\Upsilon, t\geq0\}$$
 is contained in a compact set by \eqref{eq:unif_upper_bound}, there must be an accumulation point $\omega\in\partial\RR^d_{\geq0}$ with
 \begin{equation}\label{eq:going_to_boundary}
 \inf_{\theta\in\Upsilon}\inf_{t\geq 0}\|z_\theta(t)-\omega\|_\infty=0. 
 \end{equation}
 Let $\tilde{\Sp}\subseteq\Sp$ be the species whose entries are zero in $\omega$. Note that $\tilde{\Sp}$ is not empty because $\omega\in\partial\RR^d_{\geq0}$, and for convenience denote by $\tilde{d}$ its cardinality. Consider the associated reaction network $\tilde{\G}$, as described in Lemma \ref{lem:projection}, and consider the parameter dependent time variable rate functions
 $$\tilde{\lambda}_{\tilde y\to \tilde y'}(\tilde{x},t,\theta)=\tilde{\kappa}_{\tilde y\to \tilde y'}(t,\theta)\tilde{x}^{\tilde{y}}\quad\text{for all }\tilde{x}\in\RR^{\tilde d}, \theta\in\Upsilon, \tilde y\to \tilde y'\in \tilde{\R},$$
 where 
 $$\tilde{\kappa}_{\tilde y\to \tilde y'}(t,\theta)=\sum_{\substack{y\to y'\\p(y)=\tilde y, p(y')=\tilde y'}}\kappa_{y\to y'}\frac{z_\theta(t)^y}{p(z_\theta(t))^{\tilde y}}.$$
Note that we are essentially placing the influence  of those species which are not equal to zero at $\omega$ into the (now time-dependent) rate constants.
 It follows that
 $$p(z_\theta(t))=p(z_\theta(0))+\sum_{\tilde y\to\tilde y'\in\tilde\R} (\tilde y'-\tilde y)\int_0^t\tilde\lambda_{\tilde y\to\tilde y'}(p(z_\theta(s)),s,\theta)\,ds.$$
 Moreover, $\tilde{\G}$ is strongly endotactic by Lemma \ref{lem:projection}. Note that in a neighborhood of $\omega$ the functions $\tilde{\kappa}_{\tilde y\to \tilde y'}(t,\theta)$ satisfy \eqref{eq:log_bounded} because the entries relative to species that are not in $\tilde{\Sp}$ are bounded away from 0. Hence from \eqref{eq:going_to_boundary} it follows that the solutions $p(z_\theta(\cdot))$ get arbitrarily close to the origin (of $\RR^{\tilde d}$), but this is in contradiction with the second part of Corollary \ref{cor:lyapunov} and the proof is concluded.
\end{proof}

We now state and prove here the main results of \cite{GMS:geometric}. The proofs we propose rely on Corollary \ref{cor:lyapunov} and Lemma \ref{lem:away_from_boundary}, and have substantial similarities with the techniques developed in \cite{A:single, A:boundedness, GMS:geometric}. 

\begin{theorem}\label{thm:persistence}
 Consider a deterministic mass action system $(\G, \Lambda)$, and assume $\G$ is strongly endotactic. Then, $(\G, \Lambda)$ is persistent.
\end{theorem}
\begin{proof}
 The theorem just follows from Lemma \ref{lem:away_from_boundary}, in particular from \eqref{eq:unif_lower_bound}, by considering $\Upsilon=z(0)\in\RR^d_{>0}$.
\end{proof}

\begin{theorem}\label{thm:permamence}
 Consider a deterministic mass action system $(\G, \Lambda)$, and assume $\G$ is strongly endotactic. Then, $(\G, \Lambda)$ is permanent.
\end{theorem}
\begin{proof}
Fix a set $S_y$ as in Definition \ref{def:permanence}, and let $\Gamma\subset S_y$ be as in Corollary \ref{cor:lyapunov} (the result applies if we consider the rates $\kappa_{y\to y'}(t,\theta)$  to be constant functions). Since $\Gamma\subset \RR^d_{>0}$, there exists $\varepsilon>0$ such that the enlarged set
$$\Upsilon=\{x\in S_y\,:\, \inf_{z\in\Gamma}\|x-z\|_\infty\leq\varepsilon\}\subset \RR^d_{>0}.$$
Moreover, note that $\Upsilon$ is a compact set and $\Gamma\subset\Upsilon$. Our first goal is to prove that every trajectory $\{z(t)\,:\,t\geq0\}$ with $z(0)\in S_y$ intersects $\Upsilon$. 

Let
$$\tau=\inf\left\{ t\geq 0\,:\,\frac{d}{dt}U(z(t))>0\right\}.$$  
If $\tau<\infty$, then by Corollary \ref{cor:lyapunov}, and since $\Gamma$ is compact, we have $z(\tau)\in\Gamma\subset \Upsilon$. Now suppose that $\tau=\infty$.  Since $U(\cdot)$ has a lower bound, the function $U$ can not decrease indefinitely along $z(\cdot)$.  Thus, we must have 
$$\limsup_{t\to\infty}\frac{d}{dt}U(z(t))\geq0.$$
Hence, by Corollary \ref{cor:lyapunov} and by compactness of $\Gamma$ the closure of $\{z(t)\,:\,t\geq0\}$ intersects $\Gamma$, which implies that $\{z(t)\,:\,t\geq0\}$ intersects $\Upsilon$. In conclusion, we have proved that every trajectory starting in $S_y$ intersects the compact set $\Upsilon$ at a certain finite time. Then, permanence follows from Lemma \ref{lem:away_from_boundary} by choosing
$$K=\{x\in S_y\,:\, \min_{1\leq i\leq d}x_i\geq m\text{ and }\|x\|_\infty\leq M\}$$
where
\begin{align*}
 m&=\min_{1\leq i\leq d}\inf_{z(0)\in\Upsilon}\inf_{t\geq0}z_i(t)>0\\
 M&=\sup_{z(0)\in\Upsilon}\sup_{t\geq0}\|z(t)\|_\infty<\infty.
\end{align*}
This concludes the proof.
\end{proof}

\section{Asiphonic Strongly Endotactic Networks and Large Deviation Principle}
\label{sec:LDP}

In this section, we  consider large deviations of classically scaled reaction networks. In particular, \reply{we focus on the results \cite{ADE:geometric, ADE:deviation}, which are derived from a sufficient condition that will be denoted as Assumption~\ref{4.2} later in this paper. We utilize the findings of section \ref{sec:tiers} to prove this condition holds} in a straightforward manner.

\reply{For convenience, throughout this section we denote $f(x)\approx g(x)$ or we say $f(x)$ grows like $g(x)$, if
\[
0<\lim_{\|x\|_1\to\infty} \frac{f(x)}{g(x)}<\infty.
\]}

Following \cite{K:stochastic,K:strong} we introduce the family of  classically scaled process indexed by a real number $V>0$.   In particular, we assume the process associated with $V$ is a stochastic mass action system with rate constant  $\kappa_{y\to y'}/V^{\|y\|_1 - 1}$, where $\kappa_{y\to y'}$ is a fixed positive constant.  Hence, for a particular choice of $V>0$, the intensity function for  $y\to y' \in \R$ is 
\[
	\lambda_{y\to y'}^{V}(x) = \frac{\kappa_{y\to y'}}{V^{\|y\|_1 - 1}} \mathbbm{1}_{\{x\geq y\}} \frac{x!}{(x-y)!}, \quad \text{ for } x \in \ZZ^d_{\ge 0}.
\]
We then denote the resulting stochastic process detailed in section \ref{sec:stochastic_model} by $X^V$.  Next, we  consider the scaled process
\begin{equation}\label{eq:56789998778}
\overline X^V(t)= V^{-1} X^V(t)\in V^{-1} \ZZ^d_{\ge 0}.
\end{equation}
The associated transition intensities for the process $\overline X^V$ are
\begin{equation}
\lambda^{S,V}_{y\to y'}(x)= \lambda^V_{y\to y'}(Vx) = \frac{\kappa_{y\to y'}}{V^{\|y\|_1-1}}\frac{(Vx)!}{(Vx-y)!}, \quad x \in V^{-1} \ZZ^d_{\ge 0},
\end{equation}
 and the generator  is 
\begin{equation}\label{4.1}
(\mathcal{L}_Vf)(x) =\sum_{y\to y' \in \R}\lambda^{S,V}_{y\to y'}(x)\left(f\left(x+\frac{y'-y}{V}\right)-f(x)\right), \quad x \in V^{-1} \ZZ^d_{\ge 0}.
\end{equation}

Following \cite{ADE:geometric, ADE:deviation}, we are interested in finding conditions for a reaction network to satisfy a large deviation principle (LDP).  By standard arguments, we see that for a fixed $x\in \RR^d_{>0}$ and $V$ large
\[
\lambda_{y\to y'}^{S,V}\left(\frac{\lfloor Vx \rfloor}{V}\right) =  \frac{\kappa_{y\to y'}}{V^{\|y\|_1-1}}\frac{(\lfloor Vx\rfloor)!}{(\lfloor Vx\rfloor -y)!}\approx  \frac{\kappa_{y\to y'}}{V^{\|y\|_1-1}}  V^{\|y\|_1} x^y = V\kappa_{y\to y'} x^y.
\]
Hence, we also define the analogous ``deterministic'' intensity function
\begin{equation}\label{eq:5678545678}
\lambda_{y\to y'}^{D,V}(x) = V \kappa_{y\to y'} x^y, \quad \text{ for } x \in \RR^d_{\ge 0}.
\end{equation}

 For completeness, we  provide the definition for a LDP in the setting of  reaction networks. \reply{Following the notations in \cite{ADE:deviation} and \cite{ADE:geometric}, we denote by $D_{0,T}(\RR^d_{>0})$ the Skorokhod space, or space of 
\textit{c$\grave{a}$dl$\grave{a}$g} functions $z: [0,T] \to \RR^d_{\geq 0}$, equipped with the topology of uniform convergence.
}
\begin{definition}
Fix a positive $T < \infty$ and a lower semi-continuous mapping $I: D_{0,T}(\RR^d_{>0})\to [0,\infty]$ such that for any $\alpha\in \RR_{>0}$, the level set $\{z:I(z)\leq \alpha\}$ is a compact subset of $D_{0,T}(\RR^d_{>0})$. The probability distribution of sample paths of the processes $\big\{\overline X^V\big\}_{V>0}$ with fixed initial condition $\overline X^V(0)= x\in \RR_{>0}^d$ obeys a LDP with good rate function $I(\cdot)$ if for any measurable $\Gamma \subset D_{0,T}(\RR_{>0}^d)$ we have
\begin{align*}
-\inf_{z\in \Gamma^o} I(z) &\leq \liminf_{V\to\infty}\frac{1}{V}\ln \left( P\left( \overline X^V(t)\in \Gamma \ \big| \ \overline X^V(0) = x \right)\right)\\
&\leq \limsup_{V\to\infty}\frac{1}{V}\ln \left( P \left(\overline X^V(t)\in \Gamma  \ \big| \ \overline X^V(0) = x \right) \right)\leq -\inf_{z\in\bar\Gamma}I(z)
\end{align*}
where $\Gamma^o$ and $\bar\Gamma$ denote the interior and closure of $\Gamma$ respectively.
\end{definition}
In \cite{ADE:deviation}, it is shown that under Assumption \ref{4.2} below, the process $\overline X^V$ satisfies a sample path LDP in the supremum norm. 
\begin{assumption}\label{4.2}
Let $\overline X^V$ be the process \eqref{eq:56789998778}. We assume
\begin{enumerate}
\item There exists $b< \infty$ and a continuous, positive function $U(\cdot)$ with compact sublevel sets, such that for some non-decreasing function $v': \RR_{>0} \to \RR_{>0}$,
\begin{equation}
(\mathcal{L}_V U^V)(x) \leq e^{bV} \qquad \forall V>v'(\|x\|_1), \qquad x\in V^{-1} \ZZ_{\geq 0}^d
\end{equation}
where $U^V(\cdot)$ denotes the $V^{th}$ power of $U(\cdot)$, and $\mathcal{L}_V$ is defined as in \eqref{4.1}.
\item With positive probability, starting at $\overline X^V(0)=0$, the Markov process $\overline X^V$ reaches in finite time some state $x_+$ in the strictly positive orthant $V^{-1} \ZZ_{> 0}^d$.
\end{enumerate}
\end{assumption}
Moreover, \cite{ADE:deviation} and \cite{ADE:geometric}  show that Assumption \ref{4.2} holds for reaction networks with a certain structure.  We require the following definition before stating their result.

\begin{definition}
A non-empty subset $\mathcal{P} \subset \S = \{S_1,\dots,S_d\}$ is called a \emph{siphon} if for every reaction $y\to y'\in \R$ the following condition holds:  if $y'_i >0$ for some $S_i \in \mathcal{P}$, then  $y_j >0$ for some $S_j \in \mathcal{P}$.  A reaction network is called \emph{asiphonic} if no such $\mathcal{P}$ exists. 
\end{definition}

In words, $\mathcal{P}$ is a siphon if every reaction whose product complex contains an element of $\mathcal{P}$ also has an element of $\mathcal{P}$ in its source complex.   Note that if a network is asiphonic, then $0 \in \mathcal{C}^S$ (the set of source complexes) for otherwise $\mathcal{S}$ would be a siphon.

\begin{theorem}\label{4.3}
    If the network is asiphonic and strongly endotactic (ASE), then the Markov process $\overline X^V$ satisfies Assumption \eqref{4.2} with $U$ defined as in \eqref{eq:lyapunov} (which is the usual Lyapunov function) 
and the function $v'(x)=e^x$.
\end{theorem}
Note that there is a simple argument showing that asiphonic reaction networks automatically satisfy the second part of Assumption \ref{4.2}  (see Remark 1.11 in \cite{ADE:deviation}). It is significantly harder to show ASE reaction networks satisfy  the first condition in Assumption \ref{4.2}. Here we will provide a proof  showing that ASE reaction networks satisfy the first condition of Assumption \ref{4.2}, and will do so using a tier structure argument. Specifically, we will prove Theorem \ref{4.32} below, which  implies Theorem \ref{4.3}, and is the main result of this section.
\begin{theorem}\label{4.32}
Suppose the reaction network $(\S,\C,\mathcal R)$ is ASE. Furthermore, let $U$ be defined as in \eqref{eq:lyapunov} and let $v'(x)=e^x$. Then there exists a compact set $B\subset \RR^d$ such that  for all pairs $(V,x)$ satisfying  $V>v'(\|x\|_1) = e^{\|x\|_1}$, $x\in V^{-1}\ZZ^d_{\geq 0}$, and $x \in B^c$, we have
\begin{equation}
(\mathcal{L}_V U^V)(x) < 0.
\end{equation}
\end{theorem}

Before getting to the proof of the Theorem, we need a preliminary technical result which we prove using the tier sequence technique. \reply{Later when we prove the theorem, it will be apparent that the term $H(x_n,V_n)$ in Lemma \ref{lemma:newtechnical} determines the sign of $(\mathcal{L}_V U^V)(x)$.}

\begin{lemma}\label{lemma:newtechnical}
Suppose that there is a  sequence $(x_n,V_n)_{n=0}^\infty$ such that:
\begin{align}
&\bullet \quad (x_n)_{n=0}^\infty\quad  \text{ is a tier sequence}\\
&\bullet \quad  \lim_{n\to\infty}\|x_n\|_1=\infty \label{5:condition1newnew}\hspace{3in}\phantom{.}\\
&\bullet \quad V_n> e^{\|x_n\|_1} \quad \text{ and } \quad x_n \in V_n^{-1}\ZZ^d_{> 0}.\label{5:condition2newnew}
\end{align}
Let  $c_1\in \RR$ and $c_2\in \RR_{>0}$ and let
\begin{align}\label{eq:H_SUM}
H(x_n,V_n) = \sum_{y\to y'\in \mathcal{R}}\kappa_{y\to y'} x_n^y U(x_n)\bigg(\exp\bigg(\frac{ \ln (x_n^{y'-y})+ c_1}{c_2 U(x_n)}\bigg) -1\bigg).
\end{align}
Then
\begin{align}\label{eq:567898900}
	\liminf_{n\to \infty} H(x_n,V_n) = -\infty.
\end{align}
\end{lemma}

\begin{proof}
%

Note that $U(x_n)$ grows like $\|x_n\|_1\ln( \| x_n\|_1)$, as $n \to \infty$, which itself converges to $\infty$ by \eqref{5:condition1newnew}.
Thus it must be that $\limsup_{n\to \infty} \frac{\ln (x_{n,i})}{U(x_n)} \le 0$ for each $i\in \{1,\dots,d\}$. Let us consider the set of indices 
\begin{align*}
E=\left\{i: \liminf_{n\to \infty} \frac{\ln (x_{n,i})}{U(x_n)} < 0\right\}.
\end{align*}
The set $E$ can be non-empty, and consists of the indices of those species which are  in relatively low abundance.  \reply{To make this notion more concrete, we illustrate it with an example that is not part of the proof.  Consider a two-dimensional system with $x_n = (e^{-n^2},n)$ and $V_n= e^{n^2}$.   In this case, $\ln(x_{n,1}) = -n^2$ whereas $U(x_n)$ grows like $n\ln(n)$ as $n\to \infty$.
Thus, $\lim_{n \to \infty} \frac{\ln(x_{n,1})}{U(x_n)} = -\infty$ and $1 \in E$.}

\reply{We now return to the proof.}  By potentially  considering another subsequence, we may replace all the $\liminf$ and $\limsup$ by $\lim$ in the above. Using $E$, we can partition the set of reactions $\R$ into 3 mutually exclusive groups: \reply{those that consume a species in $E$, those that produce a species in $E$ but do not consume one, and those that neither consume nor produce a species in $E$.  Specifically, we let}
\begin{itemize}
\item $\R_1=\{y\to y': y_i\neq 0 \ \text{for some} \ i\in E\}$,
\item $\R_2=\{y\to y': y_i=0 \ \forall i\in E \ \text{and} \ y'_i\neq 0 \ \text{for some} \ i\in E\}$, and
\item $\R_3=\{y\to y': y_i=y'_i=0 \ \forall i\in E\}$.
\end{itemize}
Note that because the network is asiphonic, $0 \in \C^S$.  Hence, $\R_1 \ne \R$.
We then decompose $H$ in the obvious manner as  $H(x_n,V_n)=H_1(x_n,V_n)+H_2(x_n,V_n)+H_3(x_n,V_n),$ where 
\begin{align*}
H_i(x_n,V_n)= \sum_{y\to y'\in \R_i}\kappa_{y\to y'} x_n^y U(x_n)\bigg(\exp\bigg(\frac{ \ln (x_n^{y'-y})+ c_1}{c_2U(x_n)}\bigg) -1\bigg).
\end{align*}
\reply{
We will prove that reactions from $\R_1$ give insignificant contribution to the network dynamics along a tier sequence, while there are ``good'' reactions in $\R_2$ and $\R_3$ that help stabilizing the dynamics by consuming species with high abundance}.
Specifically, we will show that (i) $\lim_{n\to \infty} H_1(x_n,V_n) = 0$,  (ii) the terms in $H_2$ are negative, and (iii)
the negative terms in $H_2$ and $H_3$ are  sufficient to guarantee that \eqref{eq:567898900} holds.

We turn to $H_1(x_n,V_n)$. First note that for $y\to y' \in\R_1$, we have that 
\[
\ln (x_n^{y' - y}) = \scal{y'}{\ln(x_n)} - \scal{y}{\ln(x_n)} \le c_3\sum_{i \in E} |\ln(x_{n,i})| = - c_3 \sum_{i\in E} \ln(x_{n,i}),
\]
for some positive constant $c_3$.
Hence, there is a $c_4>0$ so that for $n$ large enough 
\begin{align}
\label{eq:6789876new}
\begin{split}
x_n^y U(x_n) \exp\bigg(&\frac{ \ln (x_n^{y'-y})+c_1}{c_2U(x_n)}\bigg) \leq x_n^y U(x_n) \exp\bigg(-\frac{c_4 \sum_{i\in E} \ln (x_{n,i})}{U(x_n)}\bigg)\\
&= \exp\bigg(\sum_{i=1}^d y_i\ln (x_{n,i}) +\ln (U(x_n))- \frac{\sum_{i\in E}  c_4\ln (x_{n,i})}{U(x_n)} \bigg) \\
&=\exp\bigg(\sum_{i\in E}\ln (x_{n,i})\bigg(y_i-\frac{c_4}{U(x_n)}\bigg) + \sum_{j\notin E} y_j\ln (x_{n,j}) +\ln (U(x_n))\bigg).
\end{split}
\end{align}
Note that from the construction of $E$, for $i\in E$ and $j\notin E$, we must have $|\ln (x_{n,i})|\gg \ln (U(x_n))$ and $|\ln (x_{n,i})|\gg \ln (x_{n,j})$. Since $y_i\geq 1$ for some $i\in E$, we must have
\begin{align*}
\lim_{n\to\infty}\sum_{i\in E}\ln (x_{n,i})\bigg(y_i-\frac{c_4}{U(x_n)}\bigg) + \sum_{j\notin E} y_j\ln (x_{n,j}) +\ln (U(x_n))=-\infty.
\end{align*}
Moreover, by a similar argument we see that for $y\to y' \in \R_1$
\begin{equation}\label{eq:678897344}
\lim_{n\to \infty} x_n^{y} U(x_n) = \lim_{n\to \infty}  \exp\bigg(\sum_{i\in E}y_i \ln (x_{n,i}) + \sum_{j\notin E} y_j\ln (x_{n,j}) +\ln (U(x_n))\bigg) = 0.
\end{equation}
\reply{By combining all of the above it follows that} for each $y\to y'\in \R_1$
\begin{align*}
\lim_{n\to\infty}x_n^y U(x_n)&\bigg(\exp\bigg(\frac{ \ln (x_n^{y'-y})+c_1}{c_2U(x_n)}\bigg) -1\bigg) = 0
\end{align*}
and so $\lim_{n\to\infty}H_1(x_n,V_n)=0$.

Next, we consider $H_2(x_n,V_n)$. Let $y\to y' \in \R_2$.  We know that $y_j = 0$ for all $j \in E$ and that   there exist an $i\in E$ with $y'_i >0$.  Hence, using that $\lim_{n\to \infty} U(x_n) = \infty$ and  the definition of $E$, we have
\begin{align*}
\exp\bigg(\frac{\ln (x_n^{y'-y}) +c_1}{c_2U(x_n)}\bigg)-1 &=\exp\bigg(\frac{\sum_{i\in E} y_i'\ln (x_{n,i}) + \sum_{j\notin E}(y'_j-y_j)\ln (x_{n,j})+c_1}{c_2U(x_n)}\bigg)-1\\
&<e^{-c_5}-1< -c_6<0
\end{align*}
for some positive constants $c_5$ and $c_6$ and $n$ large enough. Thus
\begin{equation}\label{eq:9999new}
H_2(x_n,V_n)< -c_6\sum_{y\to y'\in \R_2}\kappa_{y\to y'} x_n^y U(x_n).
\end{equation}

We turn to  $H_3(x_n,V_n)$. Let $y\to y' \in \R_3$.  Since $y_i=y_i'=0$ for all $i\in E$, we have by the definition of $E$ that
\[
\lim_{n\to \infty}\frac{\ln (x_n^{y'-y}) +c_1}{c_2U(x_n)}=0.
\]
Note that we can choose a subsequence for which each term on the left above is either non-negative or non-positive for each $n$ and each $y\to y' \in \R_3$. 
If the terms are non-positive, we may use that $e^\rho - 1 \le \frac12 \rho$ for small $\rho \le 0$ to conclude that
\begin{align}\label{eq:nlkdajfkl;jf;akj}
\kappa_{y\to y'}& x_n^y U(x_n)\bigg(\exp\bigg(\frac{ \ln (x_n^{y'-y})+c_1}{c_2U(x_n)}\bigg) -1\bigg)\le \frac1{2c_2} \kappa_{y\to y'}x_n^y(\ln (x_n^{y'-y}) +c_1).
\end{align}
Moreover, if the terms are non-negative, we use that $e^\rho - 1 \le 2\rho$ for small $\rho \ge 0$ to conclude that 
\begin{align}\label{eq:678987654678}
\kappa_{y\to y'}& x_n^y U(x_n)\bigg(\exp\bigg(\frac{ \ln (x_n^{y'-y})+c_1}{c_2U(x_n)}\bigg) -1\bigg)\le \frac{2}{c_2} \kappa_{y\to y'}x_n^y(\ln (x_n^{y'-y}) +c_1).
\end{align}
Thus, there are positive constants $c_{y\to y'}$  for which
\begin{align}\label{eq:dfjaj;kd;afjnew}
H_3(x_n,V_n) &\leq \sum_{y\to y'\in \R_3} c_{y\to y'} \kappa_{y\to y'}x_n^y(\ln (x_n^{y'-y}) +c_1).
\end{align}

Finally, we return to $H(x_n,V_n)=H_1(x_n,V_n)+H_2(x_n,V_n)+H_3(x_n,V_n)$. 
To conclude that \eqref{eq:567898900} holds, it is now sufficient to show two things.  First,  we will prove that there is always a term in either \eqref{eq:9999new} or \eqref{eq:dfjaj;kd;afjnew} (i.e., terms associated with reactions in $\R_2$ or $\R_3$) that goes to $-\infty$, as $n\to \infty$. Second, we will prove that any positive term in the sum \eqref{eq:H_SUM} is dominated, in the sense of Lemma \ref{lem:dominated_reaction}, by a negative term.

Since the network is asiphonic, there must be a reaction for which  $0$ is the source complex. By definition of $T^{1,S}$ we have $0 \tle y$ for all $y\in T^{1,S}$, which implies that for all $y\in T^{1,S}$
\begin{equation}\label{eq:44444new}
\lim_{n\to\infty}x_n^y >0.
\end{equation}
Since the network is strongly endotactic it must be tier descending by Theorem \ref{thm:SE_tiers}. Hence there exists a reaction $y\to y' \in \R$ with $y\in T^{1,S}$ and $y'\tl y$. Recall that \eqref{eq:678897344} showed that $x_n^yU(x_n) \to 0$, as $n \to \infty$, if $y \to y' \in \R_1$.  Hence,  \eqref{eq:44444new} shows that  $y\to y' \notin \R_1$.
 If $y\to y'\in \R_2$, we consider the relevant term in \eqref{eq:9999new} and conclude 
\[
\lim_{n\to \infty} -c_6\kappa_{y\to y'} x_n^y  U(x_n) = -\infty
\]
due to the fact that $\lim_{n\to \infty}U(x_n)=\infty$. Finally, if $y\to y'\in \R_3$, we have 
\[
\lim_{n\to \infty} c_{y\to y'} \kappa_{y\to y'} x_n^y(\ln (x_n^{y'-y})+c_1) =-\infty
\]
since $y'\tl y$.  Thus, in either case, we have a term which converges to $-\infty$ as $n\to \infty$.

Next, we will show that a positive term is necessarily   dominated by a negative term.   Specifically, note that the only terms that could be  positive and not tend to zero come from the sum \eqref{eq:dfjaj;kd;afjnew} and are associated with reactions  $y\to y'\in \R_3$ with $y\tle y'$. Fix such a reaction $y \to y'\in \R_3$.    We will now show that there is necessarily a term either  in the sum \eqref{eq:9999new} or the sum \eqref{eq:dfjaj;kd;afjnew} that is negative and dominates it.

Suppose first that there is a reaction $\tilde y \to \tilde y' \in \R_2$ for which $y \tle \tilde y$.  Because $y\to y' \in \R_3$, we know
\[
U(x_n) \gg \ln(x_n^{y' - y}).
\] 
Hence, the term in  \eqref{eq:9999new} associated with $\tilde y \to \tilde y'$ dominates the positive term.

Now assume there is no such reaction $\tilde y \to \tilde y' \in \R_2$ with $y \tle \tilde y$.  
Because our network is strongly endotactic, we may apply  Lemma \ref{lem:dominated_reaction} to conclude that there exists  $y^\star\in \C$ and $y^\star\to y^{\star \star}\in\R$ such that $y\tle y^\star$, $y^{\star \star}\tl y^\star$ and for any choice of  constants $c_1',c_2'\in \RR_{>0}$ and $c_3',c_4'\in\RR$, the inequality \eqref{eq:dominated_reaction_negative} holds for $n$ large enough.  Thus, if we can show that  $y^\star \to y^{\star \star} \in \R_3$, then 
the term in  \eqref{eq:dfjaj;kd;afjnew} associated with $y^\star \to  y^{\star \star}$ dominates the positive term.

Since $y\tle y^\star$, we know from our assumption that $y^{\star} \to y^{\star \star} \notin \R_2$.  Moreover, 
since $y\tle y^\star$, the reaction $y^\star\to y^{\star\star}$ cannot be in $\R_1$ (for otherwise   the definition of $E$ and the fact that $y \to y'\in \R_3$ would imply $y^\star \ln(x_n) - y \ln(x_n)\to -\infty$, as $n \to \infty$).  
Thus, we must have  $y^\star \to y^{\star\star} \in \R_3$, and this concludes the proof of the Lemma \ref{lemma:newtechnical}.
\end{proof}

We now turn to the proof of Theorem \ref{4.3}

\begin{proof}[Proof of Theorem \ref{4.3}]
We will prove the theorem by contradiction.  We therefore suppose that there is a sequence $(x_n,V_n)_{n=0}^\infty$ such that:
\begin{align}
&\bullet \quad  \lim_{n\to\infty}\|x_n\|_1=\infty \label{5:condition1}\hspace{3in}\phantom{.}\\
&\bullet \quad V_n> e^{\|x_n\|_1} \quad \text{ and } \quad x_n \in V_n^{-1}\ZZ^d_{\geq 0}\label{5:condition2}\\
&\bullet \quad (\mathcal{L}_{V_n}U^{V_n})(x_n) \ge 0. \label{5:condition3}
\end{align}
\reply{It is important to note that the sequence $(x_n,V_n)_{n=0}^\infty$ could lie on a boundary where some species remain zero. In order to deal with those species, later in the proof we consider a modified sequence $(\tilde x_n,V_n)_{n=0}^\infty$ that lie away from the boundary but close enough to $(x_n,V_n)_{n=0}^\infty$ so that $(\mathcal{L}_{V_n}U^{V_n})(x_n)$ and $(\mathcal{L}_{V_n}U^{V_n})(\tilde x_n)$ are relatively close. This allows us to make a conclusion on $(\mathcal{L}_{V_n}U^{V_n})(x_n)$ by studying the easier object $(\mathcal{L}_{V_n}U^{V_n})(\tilde x_n)$. } Taking the zero species into consideration, and after potentially taking a subsequence, we may assume the following
\begin{enumerate}[(i)]
\item  $(x_n)_{n=0}^\infty$ is a tier sequence (this follows from Remark \ref{remark1}),
\item   there is an $\ell \in \{0,\dots,d\}$ for which  $x_{n,1} = \cdots = x_{n,\ell} = 0$ and $x_{n,j} > 0$ for all $j  \ge \ell + 1$ and all $n$  (note that $\ell$ can be zero), and 
\item   there is a  subset of the reactions, $\mathcal{P}\subseteq \R$, for which 
\begin{align}\label{eq:98047528047578205}
	\lambda^{S,V_n}_{y\to y'}(x_{n}) \begin{cases}
	> 0 & \text{ if } y\to y' \in \mathcal{P}\\
	=0 & \text{ if } y \to y' \in \R \setminus \mathcal{P}
	\end{cases}
\end{align}
for every $n$.
\item the sign of the terms $U^{V_n}(x_n)-U^{V_n}(x_n+\frac{y'-y}{V_n})$ are constant in $n$, for each $y\to y'\in\mathcal{P}$.
\end{enumerate} 
We will prove that 
$\liminf_{n\to \infty} (\mathcal{L}_{V_n} U^{V_n}) (x_n) = -\infty$, leading to a contradiction.

First, note that for any reaction $y \to y'\in \mathcal{P}$ we have
\begin{align*}
\lambda^{S,V_n}_{y\to y'}(x_n)
&=V_n\kappa_{y\to y'}\prod_{i=1}^d x_{n,i}\bigg(x_{n,i}-\frac{1}{V_n}\bigg)\dots\bigg(x_{n,i}-\frac{y_i-1}{V_n}\bigg),
\end{align*}
which is positive by assumption.  Hence, $x_{n,i} \geq \frac{y_i}{V_n}.$
Thus, for any $1 \le j \le y_i-1$,
\[
x_{n,i} -\frac{j}{V_n} =x_{n,i}-\frac{j}{y_i}\frac{y_i}{V_n}\geq x_{n,i}\bigg(1-\frac{j}{y_i}\bigg).
\]
Thus, letting $c_y=\prod_{i=1}^d\prod_{j=1}^{y_i-1}\bigg(1-\frac{j}{y_i}\bigg)>0$, we have
\begin{equation}\label{hferferjoerjo}
V_n \kappa_{y\to y'} x_n^y\geq \lambda^{S,V_n}_{y\to y'}(x_n) \geq c_{y} V_n \kappa_{y\to y'} x_n^y.
\end{equation}
Combining \eqref{hferferjoerjo} with the fact that the signs of the terms $U^{V_n}(x_n)-U^{V_n}(x_n+\frac{y'-y}{V_n})$ are constant over $n$, we may conclude that 
\begin{equation}\label{eifkernfwir934fejch}
 (\mathcal{L}_{V_n} U^{V_n}) (x_n)\leq \sum_{y\to y' \in \mathcal{P}} V_n\tilde\kappa_{y\to y'}x_n^y\left(U^{V_n}\left(x+\frac{y'-y}{V_n}\right)-U^{V_n}(x)\right)
\end{equation}
for all $n$ and for some positive constants $\tilde\kappa_{y\to y'}$, with $y \to y'\in \mathcal{P}$. For notational convenience, we define the operator
\begin{equation*}
(\mathcal{\widetilde L}_Vf)(x)=\sum_{y\to y' \in \mathcal{P}} V_n\tilde\kappa_{y\to y'}x_n^y\left(f\left(x+\frac{y'-y}{V}\right)-f(x)\right), \quad x \in V^{-1} \ZZ^d_{\ge 0},
\end{equation*}
and we point out that this operator is similar to the generator of the process $\overline X^V$ for the modified reaction rates $\tilde\kappa_{y\to y'}$. In fact, we are simply exchanging the stochastic intensities for the ``deterministic'' intensities for the reactions in $\mathcal{P}$. By \eqref{eifkernfwir934fejch}, it suffices to show that
\begin{equation}\label{detblblbl}
 \liminf_{n\to \infty} (\mathcal{\widetilde L}_{V_n} U^{V_n})(x_n) = -\infty.
\end{equation}

We consider the terms of  $(\mathcal{\widetilde L}_{V_n} U^{V_n})(x_n)$ individually.  Let $y\to y'\in \mathcal{P}$ and note that we must have $y_i = 0$ for each $i \le \ell$. 
Let
\begin{align}\label{6745erhtdhgyd87e7t}
C_{y\to y'}(V_n) = \sum_{i=1}^\ell y'_i \left(\ln \left(\frac{y'_i}{V_n}\right) -1\right).
\end{align}
Note that $|C_{y\to y'}(V_n)|$ grows at most logarithmically in $V_n$, as $n\to \infty$.
Utilizing a Taylor expansion of the logarithm yields
\begin{align*}
&U\bigg(x_n+\frac{y'-y}{V_n}\bigg)  = d+1+V_n^{-1} C_{y\to y'}(V_n)+\sum_{i=\ell + 1}^d \bigg(x_{n,i}+\frac{y'_i-y_i}{V_n}\bigg)\bigg(\ln \bigg(x_{n,i}+\frac{y'_i-y_i}{V_n}\bigg) -1\bigg)  \\
&= d+1+V_n^{-1} C_{y\to y'}(V_n) + \sum_{i=\ell+1}^d \bigg(x_{n,i}+\frac{y'_i-y_i}{V_n}\bigg)\bigg(\ln (x_{n,i})+\frac{y'_i-y_i}{x_{n,i}V_n} + r_i(x_{n,i},V_n) -1\bigg)\\
&=U(x_n)+\frac{1}{V_n}\bigg(C_{y\to y'}(V_n) + \sum_{i=\ell+1}^d (y_i' - y_i) \ln(x_{n,i}) + \sum_{i=\ell+1}^d \bigg(\frac{(y_i'-y_i)^2}{x_{n,i}V_n} + (x_{n,i}V_n+y_i'-y_i)r_i(x_{n,i},V_n)\bigg)\bigg),
\end{align*}
where 
\begin{align*}
|r_i(x_{n,i},V_n)| \leq \frac{c_1}{x_{n,i}^2V_n^2},
\end{align*}
for some $c_1>0$.  
We denote 
\begin{align*}
R_i(x_{n,i},V_n)= \frac{(y_i'-y_i)^2}{x_{n,i}V_n} + (x_{n,i}V_n+y_i'-y_i)r_i(x_{n,i},V_n).
\end{align*}

We have $x_{n,i}V_n \geq 1$ for all $i \ge \ell + 1$, thus
\begin{align}\label{eq:567898787990009new}
|R_i(x_{n,i},V_n)|\leq \frac{(y_i'-y_i)^2}{x_{n,i}V_n} + \frac{c_1}{x_{n,i}V_n} + \frac{c_1|y_i'-y_i|}{x_{n,i}^2V_n^2} \leq \frac{c_2}{x_{n,i}V_n} \leq c_2, 
\end{align}
for some positive constant $c_2$.
Combining the above, and utilizing the inequality 
\begin{align*}
(1 + \varepsilon)^n \le e^{\varepsilon n},
\end{align*}
which holds for all integers $n$ when $|\varepsilon| < 1$,
 it follows that  for $n$ large enough
\begin{align}\label{eq:45678887543new}
\begin{split}
(\mathcal{\widetilde L}_{V_n}& U^{V_n})(x_n) \\
&=\sum_{y\to y'\in \mathcal{P}} V_n\tilde\kappa_{y\to y'}x_n^y U(x_n)^{V_n}\\
&\hspace{.2in}\times \bigg(\bigg(1+\frac{1}{V_n}\frac{ C_{y\to y'}(V_n) + \sum_{i=\ell+1}^d (y_i' - y_i) \ln(x_{n,i}) + \sum_{i=\ell+1}^d R_i(x_{n,i},V_n)}{U(x_n)}\bigg)^{V_n} -1\bigg)\\
&\leq
V_nU(x_n)^{V_n-1}H_{\mathcal P}(x_n,V_n)
\end{split}
\end{align}
where
\begin{align*}
H_{\mathcal{P}}&(x_n,V_n) \\
&= \sum_{y\to y'\in \mathcal{P}}\tilde\kappa_{y\to y'} x_n^y U(x_n)\bigg(\exp\bigg(\frac{  C_{y\to y'}(V_n) + \sum_{i=\ell+1}^d (y_i' - y_i) \ln(x_{n,i}) + \sum_{i=\ell+1}^d R_i(x_{n,i},V_n)}{U(x_n)}\bigg) -1\bigg).
\end{align*}
In order to justify the inequality above, we use that (i) $\lim_{n\to\infty} U(x_n) = \infty$, (ii) the terms $R_i(x_{n,i},V_n)$ are uniformly bounded by \eqref{eq:567898787990009new}, and (iii) $\ln (x_n^{y'-y})$ is at most of order $\ln(V_n)$ because of \eqref{5:condition2} and since $x_{n,i} \ge  V_n^{-1}$ for $i \ge \ell+1$.

We will now  show that $\liminf_{n\to\infty}H_{\mathcal{P}}(x_n,V_n)=-\infty$.  To do so, we consider a new sequence $\tilde x_n$, where
\begin{equation}\label{eq:554223}
\tilde x_{n,1} = \dots =\tilde x_{n,\ell} =\frac{\alpha}{V_n}
\end{equation}
with  $\alpha =\max_{z\in \mathcal{C}, i \in \{1,\dots,d\}}{ z_i}$, and 
\[
\tilde x_{n,i}= x_{n,i}  \quad \text{for} \quad i>\ell.
\]
Because of \eqref{eq:554223} and since $u$ defined in \eqref{eq:genelized lyapunov} is a decreasing function in a positive neighborhood of zero, we have that $U(\tilde x_n) < U(x_n)$ for all $n$.  Also, since 
 $\lim_{n\to\infty}\tilde x_{n,i} =0$ for $i\leq \ell$, we have $\lim_{n\to\infty}\frac{U(\tilde x_n)}{U(x_n)}=1$.
Recalling that  $y\to y'\in\mathcal{P}$ implies  $y_i = 0$ for $i \le \ell$, we have
\begin{align}\label{dfhgdhgjdfd7758}
x_n^y=\tilde x_n^y.
\end{align}

From \eqref{6745erhtdhgyd87e7t}, and because in \eqref{eq:554223} we chose $\alpha\ge y_i'$ for all $i$,
\begin{align}\label{eq1234}
C_{y\to y'}(V_n)<\sum_{i=1}^\ell y_i'\ln  (\tilde x_{n,i}).
\end{align}
Combining  \eqref{eq1234}, $\lim_{n\to\infty}\frac{U(\tilde x_n)}{U(x_n)}=1$, and the bound on $R_i$, we may conclude  there exists $c_3\in \RR$ and $c_4\in \RR_{>0}$ such that 
\begin{align*}
\frac{ C_{y\to y'}(V_n) + \sum_{i=\ell+1}^d (y_i' - y_i) \ln(x_{n,i}) + \sum_{i=\ell+1}^d R_i(x_{n,i},V_n)}{U(x_n)} &< \frac{ \ln ( \tilde x_n^{y'-y})+ c_3}{U(x_n)}<\frac{ \ln ( \tilde x_n^{y'-y})+c_3}{c_4U(\tilde x_n)}
\end{align*}
for $n$ large enough.  
Therefore, utilizing \eqref{dfhgdhgjdfd7758} and the above yields
\begin{align}\label{eq4321}
H_{\mathcal{P}}(x_n,V_n) <\frac{U(x_n)}{U(\tilde x_n)}\sum_{y\to y'\in\mathcal{P}}\tilde\kappa_{y\to y'} \tilde x_n^y U(\tilde x_n)\bigg(\exp\bigg(\frac{ \ln ( \tilde x_n^{y'-y})+ c_3}{c_4U(\tilde x_n)}\bigg)-1\bigg).
\end{align}
By  Lemma \ref{lemma:newtechnical} we have
\begin{align}\label{0987809787078078}
\liminf_{n\to \infty} \sum_{y\to y'\in\mathcal{R}}\tilde\kappa_{y\to y'} \tilde x_n^y U(\tilde x_n)\bigg(\exp\bigg(\frac{ \ln ( \tilde x_n^{y'-y})+ c_3}{c_4U(\tilde x_n)}\bigg)-1\bigg) = -\infty.
\end{align}
Therefore, in order to conclude that $\liminf_{n\to\infty}H_{\mathcal{P}}(x_n,V_n)=-\infty$, it is sufficient to show that
\begin{align}\label{akdjfakjf;j}
\lim_{n\to \infty} \sum_{y\to y'\in\mathcal{R}\setminus\mathcal{P}}\tilde\kappa_{y\to y'} \tilde x_n^y U(\tilde x_n)\bigg(\exp\bigg(\frac{ \ln ( \tilde x_n^{y'-y})+ c_3}{c_4U(\tilde x_n)}\bigg)-1\bigg) = 0.
\end{align}

Let  $y\to y'\in \R\setminus\mathcal{P}$.  At least one of the following must be true 
\begin{enumerate}
\item there is a  $k$ with $k>\ell$ such that $y_k > 0$ and $x_{n,k}<\frac{y_k}{V_n}$.   In this case we also have  $\tilde x_{n,k}=x_{n,k} < \frac{y_k}{V_n}$.  
\item there is a  $k$ with $k\le \ell$ such that $y_k > 0$.  In this case we have $\tilde x_{n,k}=\frac{\alpha}{V_n}$.
\end{enumerate}
In either case we have  $\frac1{V_n} \le \tilde x_{n,k} \le \frac{\alpha}{V_n}$.  \reply{We select one such $k$.} Using this, together with the fact that $\ln(\|x_n\|_1) < \ln(\ln(V_n))$, implies there is a $c_5>0$ for which  
\[
\exp\bigg(\frac{ \ln ( \tilde x_n^{y'-y})+ c_3}{c_4U(\tilde x_n)}\bigg) \leq \exp\bigg(\frac{c_5 \ln V_n}{U(\tilde x_n)}\bigg) = V_n^{c_5/U(\tilde x_n)}.
\]
Thus
\begin{align}
\label{eq:09764f56787}
\begin{split}
\bigg| \tilde x_n^y U(\tilde x_n)\bigg(\exp&\bigg(\frac{ \ln ( \tilde x_n^{y'-y})+ c_3}{c_4U(\tilde x_n)}\bigg)-1\bigg) \bigg|\\
&\leq U(\tilde x_n) \bigg( \prod_{i\neq k}\tilde x_{n,i}^{y_i}\bigg) \frac{\alpha^{y_k}}{V_n^{y_k}}V_n^{c_5/U(\tilde x_n)} +  U(\tilde x_n)\bigg( \prod_{i\neq k}\tilde x_{n,i}^{y_i} \bigg) \frac{\alpha^{y_k}}{V_n^{y_k}}\\
&= U(\tilde x_n) \bigg( \prod_{i\neq k}\tilde x_{n,i}^{y_i} \bigg) \frac{\alpha^{y_k}}{V_n^{y_k-c_5/U(\tilde x_n)}} +  U(\tilde x_n)\bigg( \prod_{i\neq k}\tilde x_{n,i}^{y_i} \bigg) \frac{\alpha^{y_k}}{V_n^{y_k}}.
\end{split}
\end{align}
Since $V_n \geq e^{\|\tilde x_n\|_1}$ and $U(\tilde x_n)$ grows like $\|\tilde x_n\|_1 \ln \|\tilde x_n\|_1$, as $n\to \infty$,
    both terms go to $0$, showing \eqref{akdjfakjf;j}.  Combining \eqref{eq:45678887543new}, \eqref{eq4321}, \eqref{0987809787078078}, and   \eqref{akdjfakjf;j},  allows us to conclude that \eqref{detblblbl} holds.  Thus, the proof of the theorem is complete.
\end{proof}

\section{Network conditions for positive recurrence of strongly endotactic reaction networks}
\label{sec:PR_SE}
As we showed in \textit{Example} \ref{ex:SE_transient} and \textit{Example} \ref{ex:SE_explosive}, strong endotacticity is not a sufficient condition for positive recurrence of the associated Markov model introduced in section \ref{sec:stochastic_model}. 
Thus, in this section we provide additional network conditions for strongly endotactic reaction networks that guarantee positive recurrence. 
We note that while the previous section considered families of models under the ``classical scaling,'' this section does not.  We therefore  drop the $V$-dependence in the notation.  For example, the generator will now be denoted as $\mathcal{L}$ instead of $\mathcal{L}_V$.

We require two definitions.
\begin{definition}
A reaction network $(\S,\C,\R)$ is \emph{binary} if $\Vert y \Vert_1 \le 2$ for each $y \in \C$. 
\end{definition}

Many reaction networks in biology and chemistry are binary 
as it is rare that more than two molecules would interact simultaneously.

\begin{definition}
The reactions $0 \rightarrow S$ and $S \rightarrow 0$  are the \emph{in-flow} and \emph{out-flow} of species $S\in \S$, respectively.  A reaction network is  \emph{fully open} if $0 \to S\in \R$ and $S \to 0 \in \R$ for each $S \in \S$.
\end{definition}

The main theorem provided in this section, Theorem \ref{thm:positive recurrence 1} below, will allow us to conclude   that, for example, the Markov process associated with a reaction network that is a union of (i) a binary, strongly endotactic network, and (ii) some in-flows and all out-flows, is necessarily positive recurrent.  This is made precise in the following  corollary.

\begin{corollary}\label{thm:jinsu_binary}
Let $(\S,\C,\R)$ be a binary, strongly endotactic reaction network. Let $\R_{\text{out}}$ be the union of outflows,  $\cup_S \{S \to 0\}$, and let $\R_{\text{in}}$ be a subset of the inflows, $\cup_S\{0 \to S\}$.  Then let
\[
	\widetilde{\R}=\R\cup \R_{\text{out}} \cup \R_{\text{in}} 
\]
 and $\widetilde{\C}=\C\cup \{S\in \S\}\cup \{0\}.$
	 Then, for any choice of rate constants, the Markov process with reaction network $(\S, \widetilde\C, \widetilde\R)$ and stochastic mass action kinetics  satisfies the following: each state in a closed, irreducible component of the state space is positive recurrent; moreover, if  $\tau_{x_0}$ is the time for the process to enter the union of the closed irreducible components given an initial condition  $x_0$, then $\mathbb{E}[\tau_{x_0}] < \infty$.
\end{corollary}

Note that Corollary \ref{thm:jinsu_binary} implies that if a reaction network, $(\S,\C,\R)$ is strongly endotactic, binary, and fully open, then the associated Markov model is necessarily positive recurrent, regardless of the choice of rate constants.  This follows since in this case,  $(\S,\C,\R) = (\S, \widetilde\C, \widetilde\R)$.

We also note that when $(\S,\C,\R) \ne (\S, \widetilde\C, \widetilde\R)$ in Corollary \ref{thm:jinsu_binary},  the resulting reaction network $(\S, \widetilde\C, \widetilde\R)$ may not be strongly endotactic.  We provide an example.
\begin{example}
Consider the reaction network with species $\S = \{S_1,S_2\}$ and reactions
\[
2S_1 \rightleftarrows S_1 + S_2.
\]
This network is binary and strongly endotactic (for example, this follows because the network is weakly reversible and consists of a single linkage class \cite{GMS:geometric}).  However, the fully open network
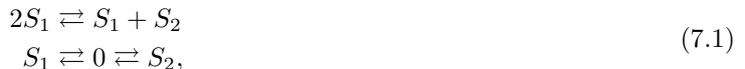
\begin{align}\label{eq:fully_open987908}
\begin{split}
2S_1 &\rightleftarrows S_1 + S_2\\
S_1 &\rightleftarrows 0 \rightleftarrows S_2,
\end{split}
\end{align}
is not strongly endotactic.  This can be seen by noting that the transversal tier sequence $(n,n)$ is not tier descending.  Of course, by Corollary \ref{thm:jinsu_binary} the fully open network \eqref{eq:fully_open987908} is positive recurrent for any choice of rate constants.  \hfill $\square$
\end{example}


Corollary \ref{thm:jinsu_binary} is a special case of Theorem \ref{thm:positive recurrence 1} below\reply{, which can be seen as a generalization to non-binary reaction networks. In Theorem \ref{thm:positive recurrence 1}, the role of $\R_{\text{out}}$ and $\R_{\text{in}}$ is played by the sets of reactions $\R'$ and $\R''$, respectively, which are defined below. Specifically, $\R'$ is a set of reactions that decrease the total number of molecules, and $\R''$ is a set of reactions that increase it. The reactions in $\R''$ are not used in the proof, and are considered in Theorem \ref{thm:positive recurrence 1} only to add to the generality of the result. The reactions in $\R'$, on the other hand, play an important role in the proof of positive recurrence, and as illustrated in Example~\ref{ex:SE_transient} the statement would not hold true without the inclusion of $\R'$.} 

 Let $(\S,\C,\R)$ be a reaction network with $\S = \{S_1,\dots,S_d\}$ and   $m = \max\{\Vert y \Vert_{1} : y \in \C^S\}$.  \reply{We remind the reader that $\C^S$ denotes the set of} source complexes.   
 Next, for  $S_i\in \S$, we let $\R_{i}$ be a nonempty, finite subset of 
\[
\left\{ a S_i \to \sum_{j=1}^d r_j' S_j : a \ge m-1 \text{ and } \sum_{j=1}^d r_j' \le a-1\right\}
\]
and let 
\[
\R' = \cup_{i = 1}^d \R_i.
\]
Let $\C'$  be the set of complexes associated with the reactions  in $\R'$.  Next, we let $\R''$ be a subset of 
\[
\left\{ \sum_{j=1}^d r_j S_j \to \sum_{j=1}^d r_j' S_j : \sum_{j=1}^d r_j \le m-2\right\}
\]
and let $\C''$ be the set of complexes associated with the reactions  in $\R''$.  Note that it is possible, though not required, that either  $\R' \subset \R$ or $\R'' \subset \R$.  It is also possible that $\R''=\emptyset$.  

\begin{theorem}\label{thm:positive recurrence 1}
Let $(\S,\C,\R)$ be a strongly endotactic reaction network with  $\S=\{S_1,S_2,\dots,S_d\}$ and let $m=\max\{\Vert y \Vert_{1} : y \in \C^S\}$.  Let $\R',\R'',\C',$ and $\C''$ be as above and let $\widetilde \R =  \R\cup\R'\cup \R''$ and $\widetilde \C =\C \cup \C' \cup \C''$.   We assume further that 
\begin{equation}\label{987987678956785657}
\max_{y\to y' \in \R'} \| y' \|_1  < \min_{y\to y'\in \R'}  \| y\|_1.
\end{equation}
Then, for any choice of rate constants, the Markov process with reaction network $(\S, \widetilde\C, \widetilde\R)$ and stochastic mass action kinetics satisfies the following: each state in a closed, irreducible component of the state space is positive recurrent; moreover, if  $\tau_{x_0}$ is the time for the process to enter the union of the closed irreducible components given an initial condition  $x_0$, then $\mathbb{E}[\tau_{x_0}] < \infty$.
\end{theorem}
Corollary \ref{thm:jinsu_binary}  follows from Theorem \ref{thm:positive recurrence 1}   by considering the case  $m = 2$.

To prove Theorem \ref{thm:positive recurrence 1}, we require the following well-known result, sometimes referred to as the ``Foster-Lyapunov criterion.'' \reply{For completeness, we include here a proof that makes use of the techniques developed in \cite{MT-LyaFosterIII}, which we refer to for more on this topic.} 

\reply{
\begin{theorem} \label{thm11}Let $X$ be a continuous-time Markov process on a state space $\mathbb{S}\subseteq\ZZ_{\geq0}^d$
with generator $\mathcal L$. Suppose there exists a finite set $K \subset \mathbb{S}$ and a function $U\colon\mathbb{S}\to\RR_{\geq0}$ such that $U(x)$ tends to infinity as $|x|\to\infty$ and 
\begin{eqnarray}\label{lya}
(\mathcal{L}U)(x) \le -1
\end{eqnarray}
for all $x \in \mathbb{S}\setminus K$. Then each state in a closed, irreducible component of $\mathbb{S}$ is positive recurrent.  Moreover, if $\tau_{x_0}$ is the time for the process to enter the union of the closed irreducible components given an initial condition  $x_0$, then $\mathbb{E}[\tau_{x_0}] < \infty$.
\end{theorem}

\begin{proof}
Non-explosivity follows from \cite[Theorem 2.1]{MT-LyaFosterIII}, hence the random variable $U(X(t))$ is well-defined for all $t>0$ and
\begin{equation}\label{aejhjhgjksahfaiWHCLKJ}
 \sup_{0\leq s\leq t}U(X(s))<\infty\quad\text{a.s.}
\end{equation}
for any $t>0$ and any initial condition $X(0)=x_0$.

To conclude the proof, by standard arguments on continuous-time Markov chains \cite{norris:markov}, it is sufficient to show that the hitting time $\tau_K$ of the finite set $K$ has finite expectation for any initial condition. For any $M\in\RR_{>0}$, let
$$
\tau_M=\inf\{t>0\,:\,U(X(t))\geq M\}.
$$
By Dynkin's formula, for any initial condition $x_0$ and any $M,t\in\RR_{>0}$
\begin{align*}
	E\left[ U(X(t \wedge \tau_K\wedge \tau_M))\right] &= U(x_0) + E \left[ \int_0^{t \wedge \tau_K\wedge \tau_M} \mathcal{L}U (X(s)) ds\right]\\
	&\le U(x_0) - E \left[  t \wedge \tau_K\wedge \tau_M \right],
\end{align*}
and by non-negativity of $U$ we have
\begin{equation}\label{dlkjfhakusehfhwg}
 E \left[  t \wedge \tau_K\wedge \tau_M \right]\leq U(x_0)<\infty.
\end{equation}
By \eqref{aejhjhgjksahfaiWHCLKJ}, for any fixed $t>0$ the random variable $t \wedge \tau_K\wedge \tau_M$ converges almost surely to $t \wedge \tau_K$, as $M\to\infty$. Therefore, by the monotone convergence theorem and \eqref{dlkjfhakusehfhwg} we have
$$ E \left[t \wedge \tau_K \right]\leq U(x_0)<\infty\quad\text{for all }t\in\RR_{>0}.$$
Since $U(x_0)$ does not depend on $t$, the stopping time $\tau_K$ is almost surely finite. Hence, $t\wedge \tau_K$ converges almost surely to $\tau_K$ as $t\to\infty$ and by applying the monotone converge theorem again we obtain that $E[\tau_K]$ is finite.
\end{proof}
}

The next lemma,   introduced in \cite[Lemma 4.1]{AK2017}, provides an  upper bound on $(\mathcal{L}U)(x_n)$, where $(x_n)_{n=0}^\infty$ is a sequence in $\mathbb{Z}^{d}_{\ge 0}$ satisfying $\lim_{n\to \infty} \|x_n\|_1 = \infty$, and $U$ is the  usual Lyapunov function defined in \eqref{eq:lyapunov}.  

\begin{lemma}\label{lemma:main}
Let $\mathcal L$ be the  generator of the Markov process associated with a reaction network $(\S,\C,\R)$ with stochastic mass-action kinetics \eqref{eq:stochastic mass action}. Let $U$ be the function defined in \eqref{eq:lyapunov}.   For a sequence $(x_n)_{n=0}^\infty$ in $\mathbb{Z}^{d}_{\ge 0}$ such that $\lim_{n\to \infty}\|x_n\|_1= \infty$, 
there is a constant $C>0$  for which
\begin{align*}
( \mathcal{L} U)(x_n) \le \sum_{y\rightarrow y'\in \R}\lambda^S_{y\to y'}(x_n)\left(\ln((x_n\vee 1)^{y'-y}) +C \right), \quad \text{for every } n \ge 0.
\end{align*}
 \end{lemma}

We will also require the following lemma in the proof of Theorem \ref{thm:positive recurrence 1}.  
\begin{lemma}\label{lem:proper x_n implies trans x_n V 1}
Let $S$ be the stoichiometric subspace of a reaction network $(\S,\C,\R)$. Let $(x_n)_{n=0}^\infty\subset \mathbb{Z}^d_{\ge 0}$ be a sequence such that $x_n-x_m \in S$ for all $n,m\in \mathbb{Z}_{\ge 0}$. If $(x_n\vee 1)_{n=0}^\infty$ is a tier sequence, then $(x_n\vee 1)_{n=0}^\infty$ is   transversal.
\end{lemma}
\begin{proof}
The proof   is essentially the same as that of Lemma \ref{lem:proper_transversal}, except $x_n$ is replaced with $x_n \vee 1$.
 \end{proof}
Now we provide the proof of Theorem \ref{thm:positive recurrence 1}.

\begin{proof}[Proof (of Theorem \ref{thm:positive recurrence 1})]
Let $\mathbb{S}$ and $S$ be the state space of the associated Markov process $X$ and stochiometric subspace of $(\S,\widetilde{\C},\widetilde{\R})$, respectively. We will show by contradiction that  \eqref{lya} holds with   $U$ defined in \eqref{eq:lyapunov}.  

Thus, we suppose that  there  exists a sequence $(x_n)_{n=0}^\infty \subset\mathbb{S}$ such that
\begin{align*}
\lmt \|x_n\|_1=\infty \quad \text{and} \quad (\mathcal L U)(x_n)\ge -1 \text{ for all }  n.
\end{align*}
By Lemmma \ref{lemma:main}, there is a positive constant $C$ such that 
\begin{align}\label{eq:upper bound of AU}
(\mathcal{L} U)(x_n) \le \sum_{y\rightarrow y'\in \widetilde{\R}} \lambda^S_{y\to y'}(x_n)\left( \ln{((x_n\vee 1)^{y'-y})} +C \right).
\end{align}
We will show that there exists a subsequence $(x_{n_k})_{k=0}^\infty$ such that

\begin{align}\label{eq:goal}
\lim_{k \to \infty} \sum_{y\rightarrow y'\in \widetilde{\R}}\lambda^S_{y\to y'}(x_{n_k})\left( \ln{((x_{n_k}\vee 1)^{y'-y})} +C \right) = -\infty,
\end{align} 
in which case the proof is completed by contradiction.

By \textit{Remark} \ref{remark1}, there must exist a subsequence $(x_{n_k}\vee 1)_{k=0}^\infty$ of $(x_n\vee 1)_{n=0}^\infty$ which is a tier sequence. By Lemma \ref{lem:proper x_n implies trans x_n V 1}, $(x_{n_k}\vee 1)_{k=0}^\infty$ is a transversal tier sequence.  Since  any subsequence of a transversal tier sequence is a transversal tier sequence, we can also assume that
\begin{enumerate}
\item for each reaction $y\to y'\in \widetilde{\R}$, either $\lambda_{y \to y'}^S(x_{n_k})\neq 0$ for all $k$ or $\lambda_{y \to y'}^S(x_{n_k})= 0$ for all $k$, 
\item for each reaction $y\to y'\in \widetilde{\R}$, we have $\lim_{k\to \infty} \lambda^S_{y\to y'} \ln{((x_{n_k}\vee 1)^{y'-y})}\in [-\infty,\infty]$ and
\item there exists an index $p \in \{1,\dots,d\}$ such that $x_{n_k,p} \ge x_{n_k,i}$ for all  $k \ge 0$ and all $i \in \{1,\dots,d\}$.
\end{enumerate}
We note that since $(x_{n_k}\vee 1)_{k=0}^\infty$ is a tier sequence, $x_{n_k,p}\to \infty$, as $k \to \infty$.   We denote by $a S_p \to y_{S_p}$ a reaction from $\R_p$.  Note that, by construction, $y_{S_p} \prec_{(x_{n_k}\vee 1)} aS_p$ and that $\lambda^S_{aS_p \to y_{S_p}} (x_{n_k} ) \to \infty$, as $k \to \infty$.

We decompose $\widetilde{\R}$ into two parts, 
\begin{align*}
&\widetilde{\R}_{\precsim_{(x_{n_k}\vee 1)}}=\{y\to y' \in \widetilde{\R} : y \precsim_{(x_{n_k}\vee 1)} y'\} \quad \text{and} \\  
&\widetilde{\Re}_{\succ_{(x_{n_k}\vee 1)}}=\{y\to y' \in \widetilde{\Re} :  y \succ_{(x_{n_k}\vee 1)}y'\}.
\end{align*} 
By \eqref{eq:upper bound of AU} we have
\begin{align}
(\mathcal L U)(x_{n_k}) \le &\sum_{y\rightarrow y'\in \widetilde{\R}_{\precsim_{(x_{n_k}\vee 1)}}}\lambda^S_{y\to y'}(x_{n_k})\left( \ln{((x_{n_k}\vee 1)^{y'-y})} +C \right) \label{eq:5477765}\\
&+
 \sum_{y\rightarrow y'\in \widetilde{\Re}_{\succ_{(x_{n_k}\vee 1)}}}\lambda^S_{y\to y'}(x_{n_k})\left( \ln{((x_{n_k}\vee 1)^{y'-y})} +C \right). \label{eq:iu69786986}
 \end{align}
For $y\to y' \in \widetilde{\Re}_{\succ_{(x_{n_k}\vee 1)}}$ we have  
\[
\lim_{k\to \infty}  \lambda^S_{y\to y'}(x_{n_k})\left( \ln{((x_{n_k}\vee 1)^{y'-y})} +C \right)=-\infty,
\]
so long as $\lambda^S_{y\to y'}(x_{n_k}) \ne 0$ for each $k$.  Moreover, $aS_p \to y_{S_p} \in  \widetilde{\Re}_{\succ_{(x_{n_k}\vee 1)}}$ and $\lambda^S_{aS_p \to y_{S_p}}(x_{n_k}) \to \infty$.  Hence, the sum in \eqref{eq:iu69786986} converges to $-\infty$, as $k \to \infty$.

Turning to \eqref{eq:5477765}, we will show that for each  $y\to y' \in \widetilde{\R}_{\precsim_{(x_{n_k}\vee 1)}}$ there exists a  $\tilde{y} \to \tilde{y}' \in \widetilde{\R}_{\succ_{(x_{n_k}\vee 1)}}$ such that for any positive constant $D$,
\begin{align}\label{eq:partial goal}
\lim_{k\to\infty}\left( \lambda^S_{y\to y'}(x_{n_k})\left ( \ln{((x_{n_k}\vee 1)^{y'-y})} +C \right) + D\lambda^S_{\tilde{y}\to \tilde{y}'}(x_{n_k})\left( \ln{((x_{n_k}\vee 1)^{\tilde{y}'-\tilde{y}})} +C \right) \right)= -\infty,
\end{align} 
where $C$ is as in \eqref{eq:upper bound of AU}, which will complete the proof.

We now fix a reaction $y \to y' \in  \widetilde{\R}_{\precsim_{(x_{n_k}\vee 1)}}$.    We have three cases, depending upon the type of reaction:
\begin{itemize}
\item \textbf{Case 1:} $y \to y' \in  \R \cap \widetilde{\R}_{\precsim_{(x_{n_k}\vee 1)}}$.
\item \textbf{Case 2:} $y \to y' \in  \R' \cap \widetilde{\R}_{\precsim_{(x_{n_k}\vee 1)}}$.
\item \textbf{Case 3:} $y \to y' \in  \R'' \cap \widetilde{\R}_{\precsim_{(x_{n_k}\vee 1)}}$.
\end{itemize}

\vspace{.1in}

\noindent \textbf{Case 1.} 
We assume  $y \to y' \in  \R \cap \widetilde{\R}_{\precsim_{(x_{n_k}\vee 1)}}$.  We apply   Lemma \ref{lem:dominated_reaction} to conclude that there exists a complex $y^\star$, a reaction $y^\star \to y^{\star \star} \in \R\cap \widetilde{\R}_{\succ_{(x_{n_k}\vee 1)}}$ for which  $y\precsim_{(x_{n_k}\vee 1)} y^{\star}$ and for which
%
%
%
\begin{align}\label{eq:Tier descending implies}
&\lim_{k\to \infty}\left(c_1(x_{n_k} \vee 1)^{y}\left( \ln{(x_{n_k} \vee 1)^{y'-y})} +c_2 \right) + c_3(x_{n_k} \vee 1)^{y^{\star}}\left( \ln{(x_{n_k} \vee 1)^{y^{\star\star}-y^\star})} +c_4 \right) \right) = -\infty, 
 \end{align}
 for any choice of constants  $c_1,c_2\in \RR_{>0}$ and $c_3, c_4 \in \RR$. 
 
 Note that if $\lambda^S_{y^{\star}\to y^{\star \star}}(x_{n_k}) \ne 0$, then \eqref{hferferjoerjo}, with $V_n = 1$, and \eqref{eq:Tier descending implies}  together imply \eqref{eq:partial goal}.
 Hence, we may assume that $\lambda^S_{y^{\star}\to y^{\star \star}}(x_{n_k}) =0$ for all $k$.  We will show that $a_pS_p \to y_{S_p}$ is the desired reaction $\tilde{y}\to \tilde{y}'$ satisfying \eqref{eq:partial goal}.
 
Since $\lambda^S_{y^{\star}\to y^{\star \star}}(x_{n_k})=0$ for all $k$, we know there is some $i \in \{1,\dots,d\}$ for which $x_{n_k,i} < y_i^\star$ for all $k$.  Hence, since $m = \max\{\|y\|_1  :  y \in \C^{\mathcal{S}}\}$,  we may conclude that 
\begin{equation}\label{eq:456789787}
(x_{n_k}\vee 1)^{y^\star} \le (y_i^\star)^{y_i^\star} x_{n_k,p}^{m-y_i^{\star}} \le  (y_i^\star)^{y_i^\star} x_{n_k,p}^{m-1}
\end{equation}
 and so
 \begin{equation*}
y^{\star}\precsim_{(x_{n_k}\vee 1)} aS_p,
 \end{equation*}
 where we recall that $a\ge m-1$.   In particular, there is a $c_5 \in \RR_{>0}$ such that for all $k$,
 \begin{equation}\label{eq:33456533}
\lambda_{aS_p \to y_{S_p}}^S(x_{n_k})=\lambda_{aS_p \to y_{S_p}}^S(x_{n_k}\vee 1) \ge c_5 (x_{n_k} \vee 1)^{y^\star}.
 \end{equation}
Turning to the logarithms, we have
\begin{align}\label{eq:5678}
\ln\left((x_{n_k}\vee 1)^{y_{S_p} - aS_p}\right) &\le \ln(x_{n,p}^{a-1}) - \ln(x_{n,p}^{a}) = -\ln(x_{n,p}).
\end{align}
 Further, for  $c_6 = - y_i^\star \ln(y_i^\star)$, 
 \begin{align}\label{eq:445787}
 \ln ((x_{n_k}\vee 1)^{y^{\star \star} - y^\star}) &\ge - \ln ( (x_{n_k}\vee1)^{y^\star}) \ge - \ln(x_{n_k,p}^{m-1}) + c_6 = -(m-1)\ln(x_{n_k,p}) + c_6,
 \end{align}
where we utilized \eqref{eq:456789787} in the final inequality.  Combining \eqref{eq:5678} and \eqref{eq:445787}   shows 
 \begin{equation}\label{almostthere}
 \ln\left((x_{n_k}\vee 1)^{y_{S_p} - aS_p}\right) \le \frac1{m-1} \left( \ln((x_{n_k}\vee 1)^{y^{\star \star} - y^\star}) - c_6\right).
 \end{equation}
 Finally, combining \eqref{eq:33456533}, \eqref{almostthere}, and \eqref{eq:Tier descending implies}  gives the desired result \eqref{eq:partial goal}, completing the proof of Case 1.

\vspace{.1in}

\noindent \textbf{Case 2.} 
We assume $y \to y' \in  \R' \cap \widetilde{\R}_{\precsim_{(x_{n_k}\vee 1)}}$. Then, by the definition of $\R'$, for some $i \in \{1,\dots,d\}$ we have $y=a' S_i$ with $a' \ge m-1$ and  $y'=\sum_{j=1}^d r_j'S_j$ with  $\Vert y'\Vert_1 = \sum_{j=1}^d r_j' \le a-1$, where we utilized \eqref{987987678956785657}.
Because we are assuming that $y \to y' \in \widetilde{\R}_{\precsim_{(x_{n_k}\vee 1)}}$, we know   $a'S_i \precsim_{(x_{n_k}\vee 1)} y'$.  
Hence, there is a $c_7 \in \RR_{>0}$ such that
\begin{align}\label{eq:compare aS_i and y'}
\lambda^S_{a'S_i\to y'}(x_{n_k})\le \kappa_{a'S_i\to y'}(x_{n_k}\vee 1)^{a'S_i}\le c_7(x_{n_k}\vee 1)^{y'} \le c_7x_{n_k,p}^{a-1}, 
\end{align}
for $k$ large enough.  
Since $\lambda^S_{a S_p\to y_{S_p}}(x_{n_k})$ is a degree $a$ polynomial in $x_{n_k,p}$, there is a constant $c_8 \in \RR_{>0}$ such that
\begin{align}\label{eq:lambda_p is degree a poly}
x_{n_k,p}^{a-1}\le c_8\frac{\lambda^S_{a S_p\to y_{S_p}}(x_{n_k})}{x_{n_k,p}} \quad \text{for each $k$.}
\end{align}
We may now combine \eqref{eq:compare aS_i and y'} and \eqref{eq:lambda_p is degree a poly} to conclude that   \eqref{eq:partial goal} holds  if we take $\tilde y\to \tilde y' = aS_p \to y_{S_p}$.  Specifically, for any $D>0$,
\begin{align*}
\lim_{k\to \infty} & \lambda^S_{aS_i\to y'}(x_{n_k})\left(\ln{((x_{n_k}\vee 1)^{y'- aS_i})}+C\right) + D\lambda^S_{aS_p\to y_{S_p}}(x_{n_k})\left(\ln{((x_{n_k}\vee 1)^{y_{S_p}- aS_p})}+C\right) \\
&\le  \lim_{k\to \infty}\lambda^S_{aS_p\to y_{S_p}}(x_{n_k})\bigg( \frac{c_7 c_8}{x_{n_k,p}}\left(\ln{(x_{n_k,p}^{a-1})}+C\right)+ D\left(\ln{((x_{n_k}\vee 1)^{y_{S_p}- aS_p})}+C\right)\bigg)\\
&= -\infty
\end{align*}
where for the last equality we used the following  facts:  (i) $\lim_{k\to \infty}\ln{((x_{n_k}\vee 1)^{y_{S_p}-aS_p})}=-\infty$ and (ii)  $\lim_{k\to \infty} \frac{1}{x_{n_k,p}}\ln(x_{n_k,p}) = 0$.

\vspace{0.2in}
\noindent \textbf{Case 3.} We assume $y \to y' \in \R''\cap \widetilde{\R}_{\precsim_{(x_{n_k}\vee 1)}}$. We will again show that \eqref{eq:partial goal} holds  if we take $\tilde y\to \tilde y' = aS_p \to y_{S_p}$.

Since $y \to y' \in \R''$, there is a constant $c_9>0$ for which
\begin{align}\label{eq:lambda_y and lambda}
\frac{\lambda^S_{y\to y'}(x_{n_k})}{\lambda^S_{aS_p\to y_{S_p}}(x_{n_k})} 
\le \frac{\kappa_{y\to y'}(x_{n_k}\vee 1)^y}{\lambda^S_{aS_p\to y_{S_p}}(x_{n_k})}
\le \frac{\kappa_{y\to y'}x_{n_k,p}^{m-2}}{\lambda^S_{aS_p\to y_{S_p}}(x_{n_k})} 
\le \frac{c_9}{x_{n_k,p}} ,\quad \text{for $k$ large enough.} 
\end{align} 
Then \eqref{eq:lambda_y and lambda} implies that for any constant $D>0$
\begin{align*}
\lim_{k\to \infty}&\left( \lambda^S_{y\to y'}(x_{n_k})\left(\ln{((x_{n_k}\vee 1)^{y'- y})}+C\right) + D\lambda^S_{aS_p\to y_{S_p}}(x_{n_k})\left(\ln{((x_{n_k}\vee 1)^{y_{S_p}- aS_p})}+C\right)\right) \\
&\le \lim_{k\to \infty}\lambda^S_{aS_p\to y_{S_p}}(x_{n_k})\left( \frac{c_9}{x_{n_k,p}}\left( \ln{(x_{n_k,p}^{\|y'\|_1})}+C\right) + D\left(\ln{((x_{n_k}\vee 1)^{y_{S_p}- aS_p})}+C\right)\right) \\
&= -\infty,
\end{align*}
where the equality follows by the same argument as the end of Case 2.  Hence, the proof is complete.
\end{proof}

\begin{example}
Now we consider the strongly endotactic reaction network $(\S,\C,\R)$ introduced in \textit{Example} \ref{ex:SE_transient}. 
\begin{align}\label{SE but transient}
0 \to 2A+B\to 4A+4B \to A,
\end{align}
As we showed in \textit{Example} \ref{ex:SE_transient}, the associated Markov process for this reaction network is transient. Note that $m=\max\{\Vert y \Vert_{1} : y \in \C^S\} = 8$ for this reaction network. We let
\begin{align}\label{eq:add reaction1}
\R_A=\{7A \to 5A+B\}, \quad \R_B=\{7B \to 6B\}, \quad \R'=\R_A\cup \R_B,
\end{align}
and $\C'=\{7A,5A+B,7B,6B\}$.
Then by Theorem \ref{thm:positive recurrence 1}, the Markov process associated to $(\S,\widetilde{C},\widetilde{R})$, where $\widetilde{C}=\C\cup \C'$ and $\widetilde{R}=\R\cup \R'$, is positive recurrent for any choice of rate constants. 

Note that we could even add extra reactions, via $\R''$, that seem to push the process away from the origin, and still reach the same conclusion.  For example,  we could let 
\begin{align}\label{eq:add reaction2}
\R''=\{6A\to 10A+10B,  \ 5A+B \to 110A+20B, \ 3A+2B \to 30B\},
\end{align}
and
\[
 \C''=\{6A, \  10A+10B, \ 5A+B,\ 110A+20B,\  3A+2B, \ 30B\}.
\]
 Then by Theorem \ref{thm:positive recurrence 1}, the Markov process associated to $(\S,\widetilde{C},\widetilde{R})$, where
 \[
 \widetilde{C}=\C\cup \C'\cup \C'' \quad \text{and} \quad \widetilde{R}=\R\cup \R'\cup \R'',
 \]
  is positive recurrent for all choice of rate constants. \hfill $\square$
\end{example}

	\bibliographystyle{plain}
	\bibliography{bib}

\begin{thebibliography}{10}

\bibitem{ADE:deviation}
Andrea Agazzi, Amir Dembo, and Jean-Pierre Eckmann.
\newblock Large deviations theory for {M}arkov jump models of chemical reaction
  networks.
\newblock {\em Ann Appl Probab}, 28(3):1821--1855, 2018.

\bibitem{ADE:geometric}
Andrea Agazzi, Amir Dembo, and Jean-Pierre Eckmann.
\newblock On the geometry of chemical reaction networks: {L}yapunov function
  and large deviations.
\newblock {\em J Stat Phys}, 172(2):321--352, 2018.

\bibitem{AM2018}
Andrea Agazzi and Jonathan Mattingly.
\newblock Seemingly stable chemical kinetics can be stable, marginally stable,
  or unstable.
\newblock Submitted, arXiv: \url{https://arxiv.org/abs/1810.06547}, 2018.

\bibitem{A:boundedness}
David~F. Anderson.
\newblock Boundedness of trajectories for weakly reversible, single linkage
  class reaction systems.
\newblock {\em J Math Chem}, 49(10):2275---2290, 2011.

\bibitem{A:single}
David~F. Anderson.
\newblock A proof of the global attractor conjecture in the single linkage
  class case.
\newblock {\em SIAM J Appl Math}, 71(4):1487--1508, 2011.

\bibitem{ACKK:explosion}
David~F. Anderson, Daniele Cappelletti, Masanori Koyama, and Thomas~G. Kurtz.
\newblock Non-explosivity of stochastically modeled reaction networks that are
  complex balanced.
\newblock 2018.

\bibitem{ACK:ACR}
David~F. Anderson, Daniele Cappelletti, and Thomas~G. Kurtz.
\newblock Finite time behavior of stochastically modeled chemical systems with
  absolute concentration robustness.
\newblock {\em SIAM J Appl Dyn Syst}, 16(3), 2017.

\bibitem{ACGW2015}
David~F. Anderson, Gheorghe Craciun, Manoj Gopalkrishnan, and Carsten Wiuf.
\newblock Lyapunov functions, stationary distributions, and non-equilibrium
  potential for reaction networks.
\newblock {\em B Math Biol}, 77(9):1744--1767, 2015.

\bibitem{ACK:product}
David~F. Anderson, Gheorghe Craciun, and Thomas~G. Kurtz.
\newblock Product-form stationary distributions for deficiency zero chemical
  reaction networks.
\newblock {\em B Math Biol}, 72(8), 2010.

\bibitem{AK2017}
David~F. Anderson and Jinsu Kim.
\newblock Some network conditions for positive recurrence of stochastically
  modeled reaction networks.
\newblock {\em SIAM J Appl Math}, 78(5):2692--2713, 2018.

\bibitem{AKbook}
David~F. Anderson and Thomas~G. Kurtz.
\newblock {\em Stochastic analysis of biochemical systems}.
\newblock Springer, 2015.

\bibitem{CJ:balance}
Daniele Cappelletti and Badal Joshi.
\newblock Graphically balanced equilibria and stationary measures of reaction
  networks.
\newblock {\em SIAM J Appl Dyn Syst}, 17(3):2146--2175, 2018.

\bibitem{CW:CB}
Daniele Cappelletti and Carsten Wiuf.
\newblock Product-form {P}oisson-like distributions and complex balanced
  reaction systems.
\newblock {\em SIAM J Appl Math}, 76(1), 2014.

\bibitem{CW:st_intermediates}
Daniele Cappelletti and Carsten Wiuf.
\newblock Elimination of intermediate species in multiscale stochastic reaction
  networks.
\newblock {\em Ann Appl Probab}, 26(5), 2016.

\bibitem{CW:det_intermediates}
Daniele Cappelletti and Carsten Wiuf.
\newblock Uniform approximation of solutions by elimination of intermediate
  species in deterministic reaction networks.
\newblock {\em SIAM J Appl Dyn Syst}, 16(4), 2017.

\bibitem{CDSS2009}
Gheorghe Craciun, Alicia Dickenstein, Anne Shiu, and Bernd Sturmfels.
\newblock Toric dynamical systems.
\newblock {\em J Symb Comput}, 44(11):1551--1565, 2009.

\bibitem{CNP2013}
Gheorghe Craciun, Fedor Nazarov, and Casian Pantea.
\newblock Persistence and permanence of mass-action and power-law dynamical
  systems.
\newblock {\em SIAM J Appl Math}, 73(1):305--329, 2013.

\bibitem{edelstein2005mathematical}
Leah Edelstein-Keshet.
\newblock {\em Mathematical models in biology}.
\newblock SIAM, 2005.

\bibitem{erdi1989mathematical}
P{\'e}ter {\'E}rdi and J{\'a}nos T{\'o}th.
\newblock {\em Mathematical models of chemical reactions: theory and
  applications of deterministic and stochastic models}.
\newblock Manchester University Press, 1989.

\bibitem{F1}
Martin Feinberg.
\newblock Complex balancing in general kinetic systems.
\newblock {\em Arch Ration Mech Anal}, 49:187--194, 1972.

\bibitem{gardiner1985stochastic}
Crispin Gardiner.
\newblock {\em Stochastic methods}, volume~4.
\newblock springer Berlin, 2009.

\bibitem{Gill76}
Daniel~T. Gillespie.
\newblock A general method for numerically simulating the stochastic time
  evolution of coupled chemical reactions.
\newblock {\em J Comput Phys}, 22(4):403--434, 1976.

\bibitem{Gill77}
Daniel~T. Gillespie.
\newblock Exact stochastic simulation of coupled chemical reactions.
\newblock {\em J Phys Chem}, 81:2340--61, 1977.

\bibitem{GMS:geometric}
Manoj Gopalkrishnan, Ezra Miller, and Anne Shiu.
\newblock A geometric approach to the global attractor conjecture.
\newblock {\em SIAM J Appl Dyn Syst}, 13(2):758--797, 2014.

\bibitem{H}
Fritz Horn.
\newblock Necessary and sufficient conditions for complex balancing in chemical
  kinetics.
\newblock {\em Arch Ration Mech Anal}, 49:172--186, 1972.

\bibitem{H-J1}
Fritz Horn and Roy Jackson.
\newblock General mass action kinetics.
\newblock {\em Arch Ration Mech Anal}, 47:187--194, 1972.

\bibitem{ingalls2013mathematical}
Brian~P Ingalls.
\newblock {\em Mathematical modeling in systems biology: an introduction}.
\newblock MIT press, 2013.

\bibitem{Kur71}
Thomas~G. Kurtz.
\newblock Limit theorems for sequences of jump {M}arkov processes approximating
  ordinary differential processes.
\newblock {\em J Appl Probab}, 8:344--356, 1971.

\bibitem{K:stochastic}
Thomas~G. Kurtz.
\newblock The relationship between stochastic and deterministic models for
  chemical reactions.
\newblock {\em J Chem Phys}, 57(7):2976--2978, 1972.

\bibitem{K:strong}
Thomas~G. Kurtz.
\newblock Strong approximation theorems for density dependent {M}arkov chains.
\newblock {\em Stoch Proc Appl}, 6:223--240, 1977/1978.

\bibitem{Kurtz80}
Thomas~G. Kurtz.
\newblock Representations of markov processes as multiparameter time changes.
\newblock {\em Ann Probab}, 8(4):682--715, 1980.

\bibitem{MT-LyaFosterIII}
Sean~P. Meyn and Richard~L. Tweedie.
\newblock {Stability of Markovian Processes III : Foster-Lyapunov Criteria for
  Continuous-Time Processes}.
\newblock {\em Adv Appl Probab}, 25(3):518--548, 1993.

\bibitem{norris:markov}
James~R. Norris.
\newblock {\em Markov chains}.
\newblock Cambridge university press, 1998.

\bibitem{Rudin:real_and_complex}
Walter Rudin.
\newblock {\em Real and complex analysis}.
\newblock McGraw-Hill, 3rd edition, 1987.

\bibitem{VK81}
N.~G. van Kampen.
\newblock {\em Stochastic processes in physics and chemistry}.
\newblock North-Holland Publishing Co., Amsterdam, 1981.
\newblock Lecture Notes in Mathematics, 888.

\end{thebibliography}
	
\end{document}